\documentclass[12pt]{amsart}
\usepackage{amsmath,amsthm,amsfonts,amssymb}
\usepackage{amssymb, amsthm, amsmath}
\usepackage[english]{babel}
\usepackage{amsfonts}
\newcommand{\ra}{\rightarrow}

\newcommand{\ot}{\otimes}
\newcommand{\mtc}{\mathcal}
\newcommand{\lam}{\lambda}
\newcommand{\Lam}{\Lambda}

\newcommand{\al}{\alpha}
\newcommand{\eps}{\epsilon}
\newcommand{\sub}{\subsection}
\newcommand{\mc}{\mathcal}
\newcommand{\D}{\Delta}
\newcommand{\teta}{\theta}
\newcommand{\ul}{\underline}
\newcommand{\mr}{\mathrm}
\newcommand{\rh}{\rightharpoonup}
\newcommand{\lh}{\leftharpoonup}
\numberwithin{equation}{section}
\newtheorem{lemma}[equation]{Lemma}
\newtheorem{thm}[equation]{Theorem}
\newtheorem{prop}[equation]{Proposition}

\newtheorem{cor}[equation]{Corollary}
\newtheorem{rem}[equation]{Remark}

\newtheorem{example}[equation]{Example}
\newcommand{\AMSclasifR}{16W35, \;16W40}

\newcommand{\ch}{\chi}
\newcommand{\mtr}{\mathrm}

\newcommand{\bn}{\begin}

\title[Drinfeld Doubles]
{A class of quantum doubles which are ribbon algebras}

\author{Sebastian  Burciu}

\address{Inst.\ of Math.\ ``Simion Stoilow" of the Romanian Academy
\\ P.O. Box 1-764, RO-014700, Bucharest, Romania, smburciu@syr.edu}

\date{September 17, 2007}
\begin{document}
\thanks{MSC (2000): \AMSclasifR}\thanks{The research was supported by CEx05-D11-11/04.10.05.}
%\vskip 0.5cm
\begin{abstract} Andruskiewitsch and Schneider classify a large class of pointed Hopf algebras with abelian coradical. The quantum double of each such Hopf algebra is investigated. The quantum doubles of a family of Hopf algebras from the above classification are ribbon Hopf algebras.
\end{abstract}
\maketitle

\section*{Introduction}

Quasitriangular Hopf algebras have an universal $R$-matrix which is a solution of the Yang-Baxter equation and their modules can be used to determine quasi-invariants of braids, knots and links. Drinfeld's quantum double construction gives a method to produce a quasitriangular Hopf algebra from a Hopf algebra and its dual.

The concept of ribbon categories was introduced by Joyal and Street. Their definition requires the notion of duality and provides isotopy invariants of framed links. Through their representations, ribbon Hopf algebras give rise to ribbon categories. They were introduced by Turaev and Reshetikhin in \cite{RYT} who also showed that the quantum groups of Drinfeld and Jimbo are ribbon algebras. A ribbon Hopf algebra is a quasitriangular Hopf algebra which possesses an invertible central element known as the ribbon element.

Kauffman and Radford \cite{RK} have shown that the Drinfeld double $D(A_l)$ of a Taft algebra $A_l$ (of dimension $l^2$) has a ribbon element if and only if $l$ is odd. The ribbon element of $D(A_l)$ for $l$ odd, provides an important invariant of $3$-manifolds (see \cite{22}). In \cite{RK} the authors also gave a criterion for a general quantum double to possess a ribbon element. Benkart and Witherspoon investigated the structure of two parameter quantum groups of $sl_n$ and $gl_n$ \cite{2wqd}. In \cite{2wrestr} they have shown that the restricted two parameter quantum groups $u_{r,\;s}(sl_n)$ are quantum doubles of certain pointed Hopf algebras and possess ribbon elements under certain compatibility conditions between the parameters $r$ and $s$.%They also studied their representation theory. \cite{2w}.

In this paper we provide a new class of quantum doubles which possess ribbon elements. They are the quantum doubles of a family of pointed Hopf algebras constructed by Andruskiewitsch and Schneider in \cite{andr}. The pointed Hopf algebras from \cite{andr} are liftings of
Radford's biproducts of Nichols algebras with group algebras.
The Radford biproducts are their associated graded algebras with respect
to the coradical filtration.
Andruskiewitsch and Schneider \cite{andr} showed that, under some restrictions
on the group order, all finite dimensional pointed Hopf
algebras having an abelian group of grouplike elements are this type of liftings.
The definition by generators and relations of these pointed Hopf algebras is very similar to that
of quantum groups and it includes Lusztig's small quantum groups.

If $G$ is an abelian finite group and $V$ a Yetter-Drinfeld module over $kG$ with a braiding of finite Cartan type ( see \cite{AS3} ) then let $A=kG\#B(V)$, where $B(V)$ is the Nichols algebra of $V$. We show that $D(A)$ is a ribbon Hopf algebra.

In Section \ref{Halg} we present the construction of the finite dimensional pointed Hopf algebras with abelian coradical constructed in \cite{andr}.

In  Section  \ref{dual} the dual Hopf algebra of such a pointed Hopf algebra is investigated. If there are no linking relations it is shown that the dual Hopf algebra contains a subalgebra isomorphic to a Nichols algebra. A pointed Hopf algebra whose root vectors are nilpotent is called a Hopf algebra of $nilpotent\; type$. In the situation of a Hopf algebra of nilpotent type and no linking relations the structure of the dual algebra is completely determined in this section. This recovers a result from \cite{MBe}. %As a consequence of this determination, it follows that the Hopf algebras of nilpotent type do not depend on the linking parameters $\lam$ (see Corollary \ref{b*}).
If the Hopf algebra is not of nilpotent type in the above sense, then its dual might not be anymore pointed and/or of nilpotent type. It will be interesting to completely determine the Hopf structure of the dual Hopf algebra in this situation. This would give new examples of Hopf algebras similar to that determined for rank one by \cite{krop}.

Section \ref{qdouble} investigates the algebra structure of the quantum double of a pointed Hopf algebra from Andruskiewitsch and Schneider's classification when there are no linking relations. In the nilpotent type situation, namely $A=kG\#B(V)$ the quantum double structure of $D(A)$ is completely determined. They have the same defining relations as the restricted two parameters quantum groups but with more grouplike elements. As an example, it is shown that for certain abelian coradical groups, the quantum double is indeed a quotient of a two parameter quantum group being isomorphic to the restricted two parameter quantum groups. This can be regarded as a generalization of the fact that a the quantum double of a Taft algebra is a a quotient of $u_{q,\;q^{-1}}(sl_2)$.

In Section \ref{braided} some notions about Hopf algebras in braided category are reminded. The integrals and distinguished grouplike elements of the bosonization algebra are given. In \cite{RK} the authors gave a criterion to decide  when a quantum double is a ribbon Hopf algebra. Using this criterion a sufficient condition for the quantum doubles of byproduct Hopf algebras to be ribbon is given.

Section \ref{Qdwr} describes the integrals and the distinguished grouplike elements for the class of Hopf algebras of nilpotent type with no linking relations, as well as for their dual Hopf algebras. It is shown that the quantum doubles corresponding to the pointed Hopf algebras of the form $B(V)\#kG$ where $V \in ^G_G\mtc{YD}$ are ribbon algebras.

The Appendix, contains some quantum binomial formulae taken from \cite{AS2} and a crucial lemma that is used in Section \ref{dual}.

Throughout this paper we work over an algebraically closed field of characteristic zero. For an abelian group $G$ and an element $g\in G$ by $<g>$ is denoted the cyclic subgroup of $G$ generated by $g$, and by $\widehat{G}$ the group of linear characters of $G$. For $g\in G$, the element $\hat{g} \in \widehat{G}$ is defined as $\hat{g}(\chi)=\chi(g)$ for all $ \chi \in \widehat{G}$.

The standard Hopf algebraic notations from \cite{montg} are used. For a Hopf algebra $A$, by $A_{_{ad}}$ is denoted the $A$-module which has the underlying vector space $A$ and for which the action of $A$ is given by the adjoint action $ad_{_{A}}(x)(y)=\sum x_1ySx_2$.

\section{The pointed Hopf algebras with abelian coradical}\label{Halg}
Let $\mathcal{D}=(G,\;(g_i)_{1 \leq i \leq \theta},\;(\chi_i)_{1 \leq i \leq \theta},\;(a_{ij})_{1 \leq i,j \leq \theta})$ be a datum of finite Cartan type associated to an abelian group $G$. That is $g_i \in G$ and $\chi_i \in \widehat{G}$ such that $\chi_i(g_i) \neq 1$ for all $1\leq i\leq \theta$ and the Cartan condition $$\chi_j(g_i)\chi_i(g_j)=\chi_i(g_i)^{a_{ij}}$$ for all $1\leq i,j \leq \theta$. The matrix $(a_{ij})_{1\leq i,j \leq \theta}$ is a Cartan matrix of finite type. Let $\Phi$ be the root system corresponding to the Cartan matrix $(a_{ij})$, $\Phi^{+}$ be the set of positive roots of the root system $\Phi$, and $\Pi=\{\al_1,\cdots,\al_{\teta}\}$ be the corresponding set of simple roots. For $\alpha_i,\;\alpha_j \in \Pi$ one writes $i \sim j$ if the corresponding nodes in the Dynkin diagram are connected. Let $\lambda=(\lambda_{ij})_{1 \leq i,j \leq \theta,\;i \nsim j}$ be a set of linking parameters, that is $\lambda_{ij} \in k$ and $$\lambda_{ij}=0,\;\text{if}\; g_ig_j=1 \;\text{or}\; \chi_i\chi_j \neq \epsilon$$

Let $V$ a finite dimensional Yetter-Drinfeld module over the group algebra $kG$. Suppose $V$ has a basis $(x_i)_{1\leq i \leq \teta}$ with $x_i \in V_{g_i}^{\chi_i}$, where $V_{g_i}^{\chi_i}:=\{gv=\chi_i(g)v,\;\delta(v)=g_i\ot v\}$ and $\delta$ is the comodule structure of $V$. The group $G$ acts by automorphisms on the tensor algebra $T(V)$ where $g(x_i)=\chi_i(g)x_i$. The braided commutators $[x_i,\;y]_c=ad_c(x_i)(y)$ are defined by $$ad_c(x_i)(y)=x_iy-g_i(y)x_i$$ for all $y\in T(V)$. The induced map $c:T(V)\ot T(V) \ra T(V)\ot T(V)$ given by $c(x_i \ot y)=g_i(y)\ot x_i$ is a braiding and $T(V)$ becomes a braided Hopf algebra in the category of Yetter-Drinfeld modules.

Andruskiewitsch and Schneider \cite{andr} introduced the following infinite dimensional Hopf algebra $U(\mtc{D},\;\lambda)$ defined by the generators $g \in G$ and $x_1,\cdots, x_{\theta}$. As an algebra, the relations in $U(\mtc{D},\;\lambda)$ are those of $G$ and $$gx_ig^{-1}=\chi_i(g)x_i$$ $$ad_c(x_i)^{1-a_{ij}}(x_j)=0\; (i \neq j,\;i \sim j)$$ $$ad_c(x_i)(x_j)=\lambda_{ij}(1-g_ig_j),\;(i <j,\; i \nsim j)$$ The coalgebra structure of $U(\mtc{D},\;\lambda)$ is given by $$\Delta(g)=g \ot g,\;\;\; \Delta(x_i)=x_i\ot1+g_i\ot x_i$$ for all $g \in G$ and $1 \leq i \leq \theta$.
Remark that $ad_c(x_i)(y)=ad(x_i)(y)$ for all $y \in A$.

Assume that the order $N_i$ of $\chi_i(g_i)$ is odd for all $i$ and is prime to $3$ for all i in a connected component of type $G_2$. The order of $\chi_i(g_i)$ is constant in each connected component $J$; denote this common order by $N_J$ or $N_{\al}$ if $\al$ is a positive root belonging to the component $J$.

For any $\alpha \in \Phi^{+}$, $\alpha=\sum_{i=1}^{i=\theta}n_i\alpha_i$, let $\mathrm{ht}(\alpha)=\sum_{i=1}^{i=\theta}n_i$. Put $$g_{_{\al}}=g_1^{n_1}\cdots g_{\teta}^{n_{\teta}} \;\; \text{and}\;\; \chi_{_{\al}}=\chi_1^{n_1}\cdots \chi_{\teta}^{n_{\teta}}.$$

Let $(\mu_{_{\alpha}})_{\al \in \Phi^{+}}$ a system of root vectors parameters, this means that $\mu_{_{\al}} \in k$ and $$\mu_{_{\al}} =0 \;\text{if}\; g_{_{\al}}^{N_{_{\al}}}=1\;\text{or}\; \chi_{_{\al}}^{N_{_{\al}}} \neq \eps .$$

Consider $(x_{_{\al}})_{\al\in \Phi^+}$ the root vectors corresponding to the positive roots $\al \in \Phi^+$. They are iterated braided commutators of $x_i$ \cite{andr}.

The finite dimensional Hopf algebra $u(\mc{D},\;\lam,\;\mu)$ is the quotient of $U(\mc{D},\;\lambda)$ by the Hopf ideal generated by $$x_{_{\al}}^{N_{_{\al}}}-u_{_{\al}}(\mu)\;\;\;(\al \in \Phi^{+})$$ where the elements $u_{_{\al}}(\mu) \in kG$ are defined in \cite{andr}. It will be later used the fact that $u_{_{\al}}(\mu)$ are central in $u(\mc{D},\;\lam,\;\mu)$ and they lie in the augmented ideal generated by $g_i^{N_i}-1$ (see \cite{andr}).

We say that $A=u(\mtc{D},\;\lam,\;\mu)$ is of $nilpotent \;type$ if $\mu_{_{\al}}=0$ for all $\al \in \Phi^{+}$. It follows from \cite{andr} that in this situation  $u_{_{\al}}(\mu)=0$ for all $\al \in \Phi^{+}$ and we shortly write $A=u(\mtc{D},\;\lam)$

Over a field of characteristic zero any pointed finite dimensional Hopf algebra with an abelian group $G$ of grouplike elements such that the order of $G$ is not divisible by primes less than $11$ is isomorphic to some $u(\mc{D},\;\lam,\;\mu)$ \cite{andr}.

\sub{PBW-bases of $U(\mc{D},\;\lam)$}
Let $y_1,\cdots,y_p$ the ordering of ${(x_{_{\al}})}_{\al\in \Phi^{+}}$ corresponding to the convex ordering $\beta_1,\cdots \beta_p$ of the positive roots. In the paper \cite{andr} it has been shown that  $\{y_1^{u_1}\cdots y_p^{u_p}g\;|\;u_i \geq 0,\;g\in G\}$ form a PBW-basis of $U(\mtc{D},\;\lam)$. The images of $y_i$ in the quotient $u(\mtc{D},\;\lam,\;\mu)$ are also denoted by $y_i$. Then $\{y_1^{u_1}\cdots y_p^{u_p}g\;|\;\; 0 \leq u_i \leq N_{\beta_i}-1,\;g \in G\}$ form a basis for $A=u(\mtc{D},\;\lam,\;\mu)$.

\sub{Grading of $U(\mc{D},\;0)$}
Let $\ul{e}_{\;1},\cdots, \ul{e}_{\;{\theta}}$ be the standard basis of $\mathbb{Z}^{\theta}$. Then $U(\mc{D},\;0)$ is a $\mathbb{Z}^{\theta}$-graded Hopf algebra \cite{andr} where the degree of $x_i$ is $\ul{e}_{\;i}$ and any group element $g\in G$ has degree zero. Given a homogeneous element $x$ in $U(\mc{D},\;\lam)$ we denote its degree by $\mr{dim}(x)$.

If $\underline{u} \in \mathbb{N}^p$, let
$$y_{_{\underline{u}}}=y_1^{u_1}\cdots y_p^{u_p},$$$$g_{{\underline{u}}}={g_{{\beta}_1}}^{u_1}\cdots {g_{{\beta}_p}}^{u_p}, $$$$\chi_{_{\underline{u}}}={\chi_{{\beta}_1}}^{u_1}\cdots {\chi_{{\beta}_p}}^{n_p}.$$ Note that if $\ul{u}=0$ then $y_{_{\ul{u}}}=g_{_{\ul{u}}}=1$ and $\chi_{_{\ul{u}}}=\eps$.

%\sub{}
For any positive root $\beta_i=\sum_{j=1}^{\teta}m_{ij}\al_j$ one has $\mtr{dim}(y_i)=\sum_{j=1}^{\teta}m_{ij}\ul{e}_j >0$ and if $\ul{u} \in \mathbb{N}^p$ then $\mtr{dim}(y_{\ul{u}})=\sum_{i=1}^pu_i\mtr{dim}(y_i)$.

Since $gx_ig^{-1}=\chi_i(g)x_i$ one has that $gy_{_{\ul{u}}}g^{-1}=\chi_{{\ul{u}}}(g)y_{_{\ul{u}}}$ for all $\ul{u} \in \mathbb{N}^p$.
From \cite{ri} one knows that if $1 \leq i<j \leq p $ then

 $$y_{j}y_i=\chi_{_{{\beta}_i}}(g_{_{{\al}_j}})y_iy_j+\sum_{I(i,\;j)}c(a_{i+1},\;\cdots,\; a_{j-1})y_{i+1}^{a_{i+1}}\cdots y_{j-1}^{a_{j-1}}$$

 where $$I(i,\;j)=\{(a_{i+1}\cdots, a_{j-1}) \in \mathbb{N}^{j-i-1}\;|\; \sum_{s=i+1}^{j-1}a_s\mathrm{dim}(y_s)=\mathrm{dim}(y_i)+\mathrm{dim}(y_j)\}$$ and $c(a_{i+1},\;\cdots a_{j-1})\in k$.

It follows that in $U(\mtc{D},\;0)$ one has \begin{equation}\label{mult}y_{_{\underline{u}}}y_{_{\underline{v}}}=\sum_{\underline{w} \in \mathbb{N}^p}y_{_{\underline{w}}}a_{\;\underline{w}}(\ul{u},\;\ul{v})\end{equation} such that $a_{\;\underline{w}}(\ul{u},\;\ul{v}) \in k$ and $\mathrm{dim}(y_{_{\underline{w}}})=\mathrm{dim}(y_{_{\underline{u}}})+\mathrm{dim}(y_{_{\underline{v}}})$.

Let \begin{equation}\label{comult}\D(y_{_{\ul{u}}})=\sum_{\ul{v},\;\ul{w} \in \mathbb{N}^p}y_{_{\ul{v}}}c_{\ul{v}\;,\ul{w}}^{\ul{u}}\ot y_{_{\ul{w}}}d_{\ul{v}\;,\ul{w}}^{\ul{u}}\end{equation} in $U(\mc{D},\;\lam)$ where $ c_{\ul{v}\;,\ul{w}}^{\ul{u}},\;d_{\ul{v}\;,\ul{w}}^{\ul{u}}\in k$. Since $U(\mc{D},\;\lam)$ is a $\mathbb{Z}^{\theta}$-graded Hopf algebra one has that $\mr{dim}(y_{_{\ul{u}}}) = \mr{dim}(y_{_{\ul{v}}})+\mr{dim}(y_{_{\ul{w}}})$.

\sub{The situation $A=u(\mtc{D},\;0,\;\mu)$}
Consider now $A=u(\mtc{D},\;0,\;\mu)$ as quotient of $U(\mtc{D},\;0)$.

Then the multiplication relation \ref{mult} becomes \begin{equation}\label{qotmult}y_{_{\underline{u}}}y_{_{\underline{v}}}=\sum_{\underline{w} \in \mathbb{N}^p}y_{_{\underline{w}}}a_{\;\underline{w}}(\ul{u},\;\ul{v})\end{equation} where  now $a_{\;\underline{w}}(\ul{u},\;\ul{v})\in kG$ and $\mathrm{dim}(y_{_{\underline{w}}}) \leq \mathrm{dim}(y_{_{\underline{u}}})+\mathrm{dim}(y_{_{\underline{v}}})$.

The comultiplication is given by \begin{equation}\label{qotcomult}\D(y_{_{\ul{u}}})=\sum_{\ul{v},\;\ul{w} \in \mathbb{N}^p}y_{_{\ul{v}}}c_{\ul{v}\;,\ul{w}}^{\ul{u}}\ot y_{_{\ul{w}}}d_{\ul{v}\;,\ul{w}}^{\ul{u}}\end{equation} in $u(\mtc{D},\;0,\;\mu)$ where now $ c_{\ul{v}\;,\ul{w}}^{\ul{u}},\;d_{\ul{v}\;,\ul{w}}^{\ul{u}}\in kG$ and $\mr{dim}(y_{_{\ul{u}}}) > \mr{dim}(y_{_{\ul{v}}})+\mr{dim}(y_{_{\ul{w}}})$.

%\sub{The situation $A=u(\mtc{D},\;0,\;\mu)$}

Let $\mc{I}$ be the ideal of $kG$ generated by the elements $u_{_{\al}}(\mu)$, $\al \in \Phi^{+}$. Then $\eps(\mc{I})=0$ and also $\chi_j(\mc{I})=0$ for any $1\leq j \leq \teta$. Indeed, the elements $u_{_{\al}}(\mu)$ lie in the augmented ideal generated by $g_i^{N_i}-1$ (see \cite{andr}) therefore $\eps(\mc{I})=0$. On the other hand $\mu_i \neq 0$ implies  that $\chi_i^{N_i}=\eps$ from the definition of $\mu_i$. The equation $\chi_j(g_i)\chi_i(g_j)=\chi_i(g_i)^{a_{ij}}$ raised to the power $N_i$ gives that $\chi_j(g_i)^{N_i}=1$, thus $\chi_j(g_i^{N_i}-1)=0$.

If $\mathrm{dim}(y_{_{\ul{w}}}) < \mathrm{dim}(y_{_{\ul{u}}})+\mathrm{dim}(y_{_{\ul{v}}})$ in \ref{qotmult} then $a_{\;\underline{w}}(\ul{u},\;\ul{v})\in \mc{I}$  since the only way to get a smaller degree in a product of type $y_{i_1}y_{i_2}\cdots y_{i_s}$ is by using the factoring relations $x_{\al}^{N_{\al}}=u_{_{\al}}(\mu)$. Then $\eps(a_{\;\underline{w}}(\ul{u},\;\ul{v}))=0$.

On the other hand if $\mr{dim}(y_{_{\ul{u}}}) > \mr{dim}(y_{_{\ul{v}}})+\mr{dim}(y_{_{\ul{w}}})$ in the comultiplication formula \ref{qotcomult} then by the same argument as above one has that $c_{\ul{v}\;,\ul{w}}^{\ul{u}} \in \mc{I}$ or $d_{\ul{v}\;,\ul{w}}^{\ul{u}} \in \mc{I}$.

In this situation $\eps (c_{\ul{v}\;,\ul{w}}^{\ul{u}})=0$ or $\eps (d_{\ul{v}\;,\ul{w}}^{\ul{u}})=0$. Moreover, since $c_{\ul{v}\;,\ul{w}}^{\ul{u}}$ or $d_{\ul{v}\;,\ul{w}}^{\ul{u}}$ is in the ideal $\mc{I}$ of $kG$ generated by $u_{_{\al}}(\mu)$ one also has $\chi_i( c_{\ul{v}\;,\ul{w}}^{\ul{u}})=0$ or $\chi_i( d_{\ul{v}\;,\ul{w}}^{\ul{u}})=0$, for all $1 \leq i \leq \teta$.

\sub{The situation $A=u(\mc{D},\;0,\;0)$}
If $A$ is of nilpotent type then the factoring relations are $x_{\al}^{N_{\al}}=0$ and the degree is preserved by multiplication or comultiplication. Thus in this situation $A=u(\mtc{D},\;0,\;0)$ is also $\mathbb{Z}^{\teta}$-graded Hopf algebra and $A\cong kG\#B(V)$ (see \cite{andr}).
\section{The dual Hopf algebra}\label{dual}
Let $A=u(\mc{D},\;0,\;\mu)$ a Hopf algebra as above. For $1\leq l \leq p$, let $\ul{f}_{\;l} \in \mathbb{N}^p$ be the element $(0,\;\cdots, 1,\; \cdots,\;0)$ with $1$ on the $l$-th position. Consider the numbers $m_i$ with $1 \leq m_i \leq p$ such that $\al_i=\beta_{m_i}$ for all $1 \leq i \leq \teta$. Thus $y_{m_i}=x_i=y_{_{\ul{f}_{\;m_i} }}$.

Extend any linear characters $\chi \in \widehat{G}$ to an element of $A^*$ such that $\chi(y_{_{\underline{u}}}g)=0$ if $\underline{u} \neq 0$. Consider also the following elements $\xi_i \in A^*$ defined by $\xi_i(y_{_{\ul{u}}}g)=\delta_{\ul{u},\;\ul{f}_{{\;m_i}}}$ for all $\ul{u} \in \mathbb{N}^p$. One has that $\xi_i(x_ia)=\eps(a)$ for all $a \in kG$.

The following lemma [\cite{krop}, Lemma 1.] will be used in the proof of the third relation of the next proposition.

\begin{lemma}\label{krp}Let $H$ be a bialgebra over the field $k$ and suppose that $K$ is a sub-bialgebra of $H$ with antipode. Suppose that $a \in G(K)$ and $x \in H\setminus K$ satisfy $xa=qax$ for some non-zero $q \in k$ and $\D(x)=x\ot a +1\ot x$. Let $V=K+Kx+\cdots$. Then:
\begin{enumerate}
\item $V$ is a free left $K$-module under left multiplication with basis $\{1,x,\; x^2,\cdots\}$ or $\{1,x,\; x^2,\cdots, x^{n-1}\}$ for some $n\geq 2$.
\item Suppose that $k$ has characteristic zero and $V$ has left $K$-module basis\\ $\{1,x,\; x^2,\cdots, x^{n-1}\}$ for some $n \geq 2$. Then $q$ is a primitive n-th root of unity and $x^n=c$ for some $c \in K$ which satisfies $\D(c)=c\ot a^n+1\ot c$. In particular $a \neq 1$.
\item Suppose that $g \in G(K)$ and $z \in K+Kx$ satisfy $\D(z)=z\ot g+1\ot z$. If $z \notin K$ then $g=a$ and $z=\alpha x+b$ where $\alpha \in K$ is not zero and $b \in K $ satisfy $\D(z)=b\ot a +1 \ot b$.
\end{enumerate}
\end{lemma}
Let $\overline{A}$ the subalgebra (with unit) of $A$ generated by $x_i,\;1 \leq i \leq \teta$.
Some algebra and coalgebra relations for $A^*$ are given in the next proposition.
\begin{prop} \label{part*}The following relations hold in $A^*$:
\begin{enumerate}
\item $\D(\xi_i)=\xi_i\ot 1+\chi_i\ot \xi_i$
\item $\chi\xi_i\chi^{-1}=\chi(g_i)\xi_i$, if $\chi \in G(A^*)$. In particular $\chi_j\xi_i=\chi_j(g_i)\xi_i\chi_j$.
\item $\xi_i^{N_i}=0$
\item $ad(\xi_i)^{1-a_{ij}}(\xi_j)=0$ for all $1 \leq i,\;j \leq \theta$
\end{enumerate}
\end{prop}
\begin{proof}

 1) From definition of $\xi_i$ it can be seen that $\xi_i(yg)=\xi_i(y)$ for all $y \in \overline{A}$ and $g \in G$. One has to show that $\xi_i(ab)=\xi_i(a)\eps(b)+\chi_i(a)\xi_i(b)$ for all $a,\;b\in A$. It is enough to check the last relation on the basis elements of $A$. Thus one has to show that: \begin{equation}\label{tmb}\xi_i((y_{_{\ul{u}}}g)(y_{_{\ul{v}}}h))=\xi_i(y_{_{\ul{u}}}g)
 \eps(y_{_{\ul{v}}}h)+\chi_i(y_{_{\ul{u}}}g)\xi_i(y_{_{\ul{v}}}h)\end{equation}
for all $\ul{u},\;\ul{v}\in \mathbb{N}^p$ and all $g,\;h \in G$.

Since $gy_{_{\ul{v}}}=\chi_{_{\ul{v}}}(g)y_{_{\ul{v}}}g$ it follows that $\xi_i((y_{_{\ul{u}}}g)(y_{_{\ul{v}}}h))=\chi_{_{\ul{v}}}(g)\xi_i(y_{_{\ul{u}}}y_{_{\ul{v}}}gh)=\\=\chi_{_{\ul{v}}}(g)\xi_i(y_{_{\ul{u}}}y_{_{\ul{v}}})$. On the other hand $\xi_i(y_{_{\ul{u}}}g)\eps(y_{_{\ul{v}}}h)+\chi_i(y_{_{\ul{u}}}g)\xi_i(y_{_{\ul{v}}}h)=\xi_i(y_{_{\ul{u}}})\eps(y_{_{\ul{v}}})++
\chi_i(y_{_{\ul{u}}}g)\xi_i(y_{_{\ul{v}}})$. Thus one has to show that:  \begin{equation*}\chi_{_{\ul{v}}}(g)\xi_i(y_{_{\ul{u}}}y_{_{\ul{v}}})=\xi_i(y_{_{\ul{u}}})\eps(y_{_{\ul{v}}})+
\chi_i(y_{_{\ul{u}}}g)\xi_i(y_{_{\ul{v}}})\end{equation*}

 If $\ul{u} \neq 0$ and $\ul{v} \neq 0$ then $\mathrm{dim} (y_{_{\ul{u}}}) >0$ and $\mathrm{dim}(y_{_{\ul{v}}})>0$. The right hand side of the above equation is zero since $\eps(y_{_{\ul{v}}})=\chi_i(y_{_{\ul{u}}}g)=0$. On the other hand if $y_{_{\ul{u}}}y_{_{\ul{v}}}$  written with respect to the above basis of $A$ contains a term of the type $x_ia_i$ with $a_i\in kG$  then since $\mr{dim}(y_{_{\ul{u}}}y_{_{\ul{v}}}) \neq \mr{dim}(x_i)$ it follows from the discussion of the previous section that $\eps(a_i)=0$ and then $\xi_i(x_ia_i)=0$. Thus in this situation both terms of the above equation are zero. (Note that $\mr{dim}(y_{_{\ul{u}}}y_{_{\ul{v}}})=\mr{dim}(x_i)$ implies that $\ul{u}=\ul{f}_{m_i}$ and $\ul{v}=0$ or $\ul{u}=0$ and $\ul{v}=\ul{f}_{m_i}$.)

 Suppose now that $\ul{u} = 0$ which means that $y_{_{\ul{u}}} =1$. The  equation \ref{tmb} becomes $\chi_{_{\ul{v}}}(g)\xi_i(y_{_{\ul{v}}})=\chi_i(g)\xi_i(y_{_{\ul{v}}})$. From the definition of $\xi_i$ the only possibility for both terms to be nonzero is that of $\ul{v}=\ul{f}_{\;i}$ which means $y_{_{\ul{v}}}=x_i$. In this situation the left hand side is $\chi_{_{\ul{f}_{\;i}}}(g)\xi_i(x_i)=\chi_i(g)$ which is the same value as the one of the right hand side term.

 The last possibility to discuss is when $\ul{v}=0$ which means that $y_{_{\ul{v}}} =1$. Then the equation \ref{tmb} becomes $\xi_i(y_{_{\ul{u}}})=\xi_i(y_{_{\ul{u}}})$.

 2) If $\chi \in G(A^*)$ then $\chi(u_{_{\al}}(\mu))=\chi(x_{\al}^{N_{\al}})=\chi(x_{\al})^{N_{\al}}=0$. Since $\chi(u_{_{\al}}(\mu))=0$ it follows that $\chi$ is zero on the ideal $\mc{I}$ of $kG$. %In particular for $\al$ a simple root $\al_i$ one has $\mu_i\chi(g_i^{N_i}-1)=0$.

One has to prove that \begin{equation}\label{cai}\chi\xi_i(y_{_{\ul{u}}}g)=\chi(g_i)\xi_i\chi(y_{_{\ul{u}}}g)\end{equation} for all $\ul{u}\in  \mathbb{N}^p$ and $g \in G$.

As in the previous section, let $$\D(y_{_{\ul{u}}})=\sum_{\ul{v},\;\ul{w} \in \mathbb{N}^p}y_{_{\ul{v}}}c_{\ul{v}\;,\ul{w}}^{\ul{u}}\ot y_{_{\ul{w}}}d_{\ul{v}\;,\ul{w}}^{\ul{u}}$$ where $ c_{\ul{v}\;,\ul{w}}^{\ul{u}},\;d_{\ul{v}\;,\ul{w}}^{\ul{u}}\in kG$. Then the first term of equation \ref{cai} becomes $$\chi\xi_i(y_{\ul{u}}g)=\sum_{\ul{v},\;\ul{w} \in \mathbb{N}^p}\chi(y_{_{\ul{v}}}c_{\ul{v},\;\ul{w}}^{\ul{u}}g)\xi_i( y_{_{\ul{w}}}d_{\ul{v}\;,\ul{w}}^{\ul{u}}g)$$ The only possibility for the right hand side term of the previous equality to be nonzero is when $\mr{dim} (y_{_{\ul{v}}})=0$ and $\mr{dim}( y_{_{\ul{w}}})=\ul{e}_{\;i}$ which means $\ul{v}=0$ and $\ul{w}=\ul{f}_{\;i}$. If $\mr{dim}(y_{_{\ul{u}}}) \neq \ul{e}_{\;i}$ then this is possible only by reduction via the factoring relations and as in the discussion from the previous section it follows that either $c_{\ul{v},\;\ul{w}}^{\ul{u}}$ or $d_{\ul{v},\;\ul{w}}^{\ul{u}}$ are in the ideal $\mc{I}$ generated by $u_{_{\al}}(\mu)$. Then either $\chi(y_{_{\ul{v}}}c_{\ul{v},\;\ul{w}}^{\ul{u}}g)=0$ (if $c_{\ul{v},\;\ul{w}}^{\ul{u}} \in \mc{I}$) or $\xi_i( y_{_{\ul{w}}}d_{\ul{v}\;,\ul{w}}^{\ul{u}}g)=0$ (if $d_{\ul{v},\;\ul{w}}^{\ul{u}}\in \mc{I}$). Thus if $\mr{dim}(y_{_{\ul{u}}}) \neq \ul{e}_{\;i}$ the left hand side of the equation \ref{cai} is zero.

If $\mr{dim}(y_{_{\ul{u}}})= \ul{e}_{\;i}$, which is equivalent to $y_{_{\ul{u}}}=x_i$, then $\D(x_ig)=x_ig\ot g+ g_ig \ot x_ig$ and $\chi\xi_i(x_ig)=\chi(g_ig)$.

For the second term of equation \ref{cai} one has that $$\chi(g_i)\xi_i\chi(y_{_{\ul{u}}}g)=\chi(g_i)\sum_{\ul{v},\;\ul{w} \in \mathbb{N}^p}\xi_i(y_{_{\ul{v}}}c_{\ul{v},\;\ul{w}}^{\ul{u}}g)\chi( y_{_{\ul{w}}}d_{\ul{v}\;,\ul{w}}^{\ul{u}}g)$$ A similar discussion shows that the only possibility for this term to be nonzero is when $\mr{dim}(y_{_{\ul{v}}})=\ul{e}_{\;i}$ and $\mr{dim}(y_{_{\ul{w}}})=0$ which are equivalent to $\ul{v}=\ul{f}_{\;m_i}$ and $\ul{w}=0$. If $\mr{dim}(y_{_{\ul{u}}}) \neq \ul{e}_{\;i}$ then as in the discussion from the previous paragraph it follows that either $c_{\ul{v},\;\ul{w}}^{\ul{u}}$ or $d_{\ul{v},\;\ul{w}}^{\ul{u}}$ are in the ideal $\mc{I}$ generated by $u_{_{\al}}(\mu)$ and then the value of the term is still $0$.

If $y_{_{\ul{u}}}=x_i$ then, using the formula for $\D(x_i)$, %=x_ig\ot g+ g_ig \ot x_ig$ and
one has that $\chi(g_i)\xi_i\chi(x_ig)=\chi(g_ig)$, thus the equation \ref{cai} is true in this situation too.

Computing $(\D \ot \mr{Id})\D$ and $(\mr{Id} \ot \D )\D$ for $\xi_i$ in the formula from 1) it follows that $\D(\chi_i)=\chi_i\ot \chi_i$, thus $\chi_i$ are grouplike elements of $A^*$ for any $1 \leq i \leq \teta$. Then the second relation of 2) follows from the first one.

3)Let $H$ be the Hopf subalgebra of $A^{*\;coop}$ generated by $\xi_i$ and $\chi_i$ %\in \widehat{G}$.
One has $\chi_i\xi_i=\chi_i(g_i)\xi_i\chi_i$ and the order of $\chi_i(g_i)$ is $N_i$. The second statement of Lemma \ref{krp} applied for $K=k<\chi_i>$ and  $x=\xi_i$ gives that $\xi_i^{N_i} \in k<\chi_i>$. But since $\xi_i^{N_i}(g)=0$ for all $g \in G$ it follows that $\xi_i^{N_i}=0$.

4) Let $z=ad(\xi_i)^{1-a_{ij}}(\xi_j)$. Clearly $z(g)=0$ for all $g \in G(A)$ since $\xi_j(g)=0$. From Lemma \ref{prim} from Appendix one knows that $z$ is a skew primitive element of $A^*$, that is $$\D(z)=z \ot 1+ \chi \ot z$$ where $\chi=\chi_i^{1-a_{ij}}\chi_j$.  Then $z(gy)=\chi(g)z(y)$ for all $g \in G$ and $y \in \overline{A}$. On the other hand $z(x_ix_j)=z(x_i)\eps(x_j)+\chi(x_i)z(x_j)=0$ for all $1 \leq i,\;j \leq \teta$ and by induction on $r$ one has $z(x_{i_1}x_{i_2}\cdots x_{i_r})=0$ for all $r \geq 2$.

In order to show that $z=0$ it is enough to check that $z(x_m)=0$ for all $1 \leq m\leq \teta$.

Let $f,f' \in A^*$. Then $$(ad(f)(f'))(x)=(f_1f'S(f_2))(x)=f_1(x_1)f'(x_2)f_2(Sx_3)$$ for all $x \in A$. Since $$\D^2(x_m)=x_m \ot 1 \ot 1+ g_m \ot x_m \ot 1 + g_m \ot g_m \ot x_m$$ one has
\begin{equation*}(ad(f)(f'))(x_m)  =  f_1(x_m)f'(1)f_2(1)+f_1(g_m)f'(x_m)f_2(1)+f_1(g_m)f'(g_m)f_2(Sx_m)=\end{equation*}
\begin{equation*}=f(x_m)\eps(f')+f(g_m)f'(x_m)+f(g_mS(x_m))f'(g_m)\end{equation*}
Suppose moreover that $f(g_m)=f'(g_m)=0$ and $\eps(f')=0$. Then $(ad(f)(f'))(x_m)=0$. Clearly $f=\xi_i$ and $f'=ad(\xi_i)^{-a_{ij}}(\xi_j)$ satisfy the above conditions, thus $z(x_m)=0$.
\end{proof}
%\sub{ Construction of $Y_{\al}$.}
\begin{prop}Let $A=u(\mc{D},\;0,\;\mu)$ as above and $H$ be the subgroup of $G$ generated by the elements $<g_i^{N_i}\;|\;\mu_i\neq 0>$. Then $G(A^*)=\widehat{G/H}$.
\end{prop}
\begin{proof}
If $\ch \in G(A^*)$ then relation $gx_ig^{-1}=\ch_i(g)x_i$ implies $\ch(x_i)=0$ for all $1 \leq i \leq \theta$. Thus $\ch(\mu_i(g_i^{N_i}-1))=\ch(x_i^{N_i})=0$ and since $\mu_i \neq 0$ it follows that $\ch(g_i^{N_i})=1$ and $\ch \in \widehat{G/H}$. Conversely, suppose $\ch \in \widehat{G/H} \subset \widehat{G}$ and extend $\ch$ to an element in $A^*$ as at the beginning of this section. The equation $\ch(ab)=\ch(a)\ch(b)$ will be verified on the basis elements of $A$. Suppose $a=y_{_{\ul{u}}}g$ and $b=y_{_{\ul{v}}}h$.
If $\ul{u}\neq 0$ or $\ul{v}\neq 0$ then clearly $\ch(a)\ch(b)=0$. On the other hand $\ch(ab)=\ch(y_{_{\ul{u}}}y_{_{\ul{v}}}gh)\ch_{ _{\ul{v}}}(g)=0$
since the part of degree zero of the product $y_{_{\ul{u}}}y_{_{\ul{v}}}$ is in $\mtc{I}$ and by its definition $\ch|_{ _{\mtc{I}}}=0$.
If $\ul{u}=0$ and $\ul{v}=0$ then the equation $\ch(ab)=\ch(a)\ch(b)$ is satisfied since $\ch$ is a character of $G$.
\end{proof}
Let $\overline{A^*}$ the subalgebra of $A^*$ generated by $(\xi_i)_{1 \leq i \leq \teta}$. It follows that $\overline{A^*}$ is the Nichols algebra of the $\widehat{G}$ braided vector space $W$ with basis given by $Y_i \in W^{\hat{g_i}}_{\chi_i}$.
Similarly to the construction for $A$, for any $\al \in \Phi^{+}$ let $Y_{\al}$ be the corresponding iterated commutators of $\xi_i$. Denote these elements with $Y_1,\cdots, Y_p$ using the convex ordering of the positive roots. Clearly $Y_{m_i}=\xi_i$ for all $\leq i \leq \teta$.

Let $\mc{\widetilde{D}}=(\widehat{G},\;(\chi_i)_{1 \leq i \leq \teta},\;{(\hat{g_i})}_{1 \leq i \leq \teta}),\;(a_{ij})_{1 \leq i,j \leq \teta})$. It can be verified that is a datum of finite Cartan type associated to the abelian group $\widehat{G}$.
\begin{cor}\label{b*} If $A \cong u(\mc{D},\;0,\;0)$ is a pointed Hopf algebra of nilpotent type then $A^*=u(\mc{\widetilde{D}},\;0,\;0)$ is also a pointed Hopf algebra. A basis for $A^*$ is given by $\{\chi Y_{_{\ul{u}}}\; |\;\;\ul{u} \in \mathbb{N}^p, \;0 \leq u_i \leq N_{\beta_i}-1,\;\chi \in \widehat{G}\}$.
\end{cor}
\begin{proof} If $\mu=0$ and $\lam=0$ then $A$ is a $\mathbb{Z}^{\teta}$-graded Hopf algebra and any $\chi \in \widehat{G}$ extended to $A^*$ as in the beginning of this section becomes a grouplike element of $A^*$. The previous theorem implies that $A^*\cong k\widehat{G}\#B(W)$ and the basis description follows from \cite{andr}.
\end{proof}

\begin{rem} In a recent paper \cite{MG} it was proved that all the liftings of $B(V)\#kG$ where $G$ is an abelian group whose order has no prime divisors $< 11$ are monoidally Morita-Takeuchi equivalent and therefore cocycle deformations of $B(V)\#kG$. Thus their dual algebras are the same and Corollary \ref{b*} remains true for any lifting of $B(V)\#kG$. Thus if $A=u(\mc{D},\;\lam,\;\mu)$ then $A^* \cong u(\mc{\widetilde{D}},\;0,\;0)$ as algebras for any $\lam$ and $\mu$.
\end{rem}
\section{The quantum double of $A$}\label{qdouble}

 Let $A=u(\mc{D},\;0,\;\mu)$ as in the previous section.
\begin{prop} \label{defning}The following relations hold in $D(A)$:
\begin{enumerate}
\item $g\xi_ig^{-1}=\chi_i^{-1}(g)\xi_i$ for all $g \in G$.
\item $g\gamma=\gamma g$ for any $g \in G$ and $\gamma \in \widehat{G}$.
\item $x_i\xi_j=\xi_jx_i$ for $i \neq j$.
\item $[x_i,\; \xi_i]=\chi_i-g_i$ for all $1 \leq i \leq \teta$.
\item If $\gamma \in G(A^*)$ then $\gamma^{-1}x_i\gamma=\gamma(g_i) x_i$ for all $1 \leq i \leq \teta$.
\end{enumerate}
\end{prop}

\begin{proof}
One has that $$af=(a_1 \;\; \rightharpoonup\;\; f \;\;\leftharpoonup \;\;S^{-1}a_3)a_2$$
%=p_2(({S^{-1}}^*(p_1) \;\;\\rightharpoonup h \leftharpoonup p_3)$$
for all $a \in A$ and $f \in A^*$. For the first formula notice that $g\xi_i=(g \rh \xi_i \lh g^{-1})g$ and $g \rightharpoonup \xi_i \leftharpoonup g^{-1} =\chi_i^{-1}(g)\xi_i$. Similarly $g\gamma=(g \rh \gamma \lh g^{-1})g$ and $g \rh \gamma=\gamma(g)\gamma$ while $\gamma \lh g^{-1}=\gamma(g^{-1})\gamma$. Thus the second formula is proved.

To prove relations 3) and 4) notice that $$\Delta^{2}(x_i)=g_i \ot x_i \ot 1 + x_i \ot 1 \ot 1+g_i \ot g_i \ot x_i$$ Then $x_if= (g_i \rightharpoonup f)x_i+x_i \rightharpoonup f+(g_i \rightharpoonup f \leftharpoonup S^{-1}x_i)g_i$, for all $f \in A^*$.

Since $S^{-1}x_i=-x_ig_i^{-1}$ this last formula becomes \begin{equation}\label{harp}x_if= (g_i \rightharpoonup f)x_i+x_i \rightharpoonup f-(g_i \rightharpoonup f \leftharpoonup x_i\lh g_i^{-1})g_i\end{equation}
If $f=\xi_j$ with $j \neq i$ then $g_i \rh \xi_j=\xi_j$ and the first term of the above  equality is $\xi_jx_i$. On the other hand the other two terms are zero since $x_i\rh \xi_j=\xi_j \lh x_i=0$. Indeed $(x_i \rh \xi_j)(y_{_{\ul{u}}}g)=\xi_j(y_{_{\ul{u}}}gx_i)=\chi_i(g)\xi_j(y_{_{\ul{u}}}x_ig)=\chi_i(g)\xi_j(y_{_{\ul{u}}}x_i)$. Since $i \neq j$ one has $\mtr{dim}(y_{_{\ul{u}}}x_i) \neq \mtr{dim}(x_j)$. Then the product $y_{_{\ul{u}}}x_i$ has a term of the type $x_ja_j$ with $a_j \in kG$ in its writing as linear combination of the standard basis (after putting all terms $x_jg$ together) only by using the factorizing relations. Thus in this situation $a_j \in \mc{I}$ and $\eps(a_j)=0$ which implies that $\xi_j(x_ja_j)=0$. Similarly, $\xi_j \lh x_i=0$ and the third relation is proved.

For the next relation suppose that $f=\xi_i$. Then $g_i \rh \xi_i=\xi_i$ and the first term of the above equality is $\xi_ix_i$. On the other hand $x_i \rh \xi_i=\chi_i$ since $(x_i \rh \xi_i)(y_{_{\ul{u}}}g)=\xi_i(y_{_{\ul{u}}}gx_i)=\chi_i(g)\xi_i(y_{_{\ul{u}}}x_ig)=\chi_i(g)\xi_i(y_{_{\ul{u}}}x_i)$ and if $\ul{u} \neq 0$ (which means that $\mr{dim}( y_{_{\ul{u}}}) \neq  0$) then as before this term is zero. If $\ul{u}= 0$ which means $ y_{_{\ul{u}}}=1$ then $(x_i \rh \xi_i)(y_{_{\ul{u}}}g)=(x_i \rh \xi_i)(g)=\xi_i(gx_i)=\chi_i(g)$. Thus the second term of equation \ref{harp} is $\chi_i$. The last term, $-(g_i \rh \xi_i \lh x_i \lh g_i^{-1})g_i$ is equal to -$g_i$ since $g_i \rh \xi_i=\xi_i$, $\xi_i \lh x_i=\eps$ and $\eps \lh g_i^{-1}=\eps$. The proof for $\xi_i \lh x_i=\eps$ is similar to the one of $x_i \rh \xi_i=\chi_i$. The proof of 4 is now complete.

For the last relation put $f=\gamma$ in \ref{harp}. One has $g_i \lh \gamma =\gamma(g_i)\gamma$. The other two terms are zero since $x_i \rh \gamma=\gamma \lh x_i=0$. The proof of these facts is similar to the one in part 3). One uses that $\gamma$ is zero on $\mc{I}$ since $\gamma \in G(A^*)$.
\end{proof}

%\begin{prop} The following relations hold in $D(A)$:\begin{enumerate}\end{enumerate}\end{prop}

Let $A = u(\mc{D},\;0,\;0)$ with $\mc{D}=(G,\;(g_i)_{1 \leq i \leq \theta},\;(\ch_i)_{1 \leq i \leq \theta},\;(a_{ij})_{1 \leq i,j \leq \theta})$ a Cartan datum of finite type. Using Proposition \ref{defning} and formula \ref{newserre} from  Appendix the following relations hold in $D(A)$:
\begin{equation}x_i^{N_i}=(\xi_i\ch_i^{-1})^{N_i}=0\end{equation}
\begin{equation}(g\ch)x_i(g\ch)^{-1}=<\ch_i\hat{g_i}^{-1},\;g\ch>x_i\end{equation}
\begin{equation}(g\ch)(\xi_i\ch_i^{-1})(g\ch)^{-1}=<\ch_i^{-1}\hat{g_i},\;g\ch>\xi_i\end{equation}
\begin{equation}ad(x_i)^{1-a_{ij}}(x_j)=0\end{equation}
\begin{equation}ad(\xi_i\ch_i^{-1})^{1-a_{ij}}(\xi_j\ch_j^{-1})=0\end{equation}
\begin{equation} \label{linkingrel}ad(x_i)(\xi_i\ch_i^{-1})=(1-g_i\ch_i^{-1})\end{equation}
\begin{equation} \D(x_i)=x_i \ot 1+g_i\ot x_i\end{equation}
\begin{equation}\D(\xi_i\ch_i^{-1})=\ch_i^{-1}\ot \xi_i\ch_i^{-1}+\xi_i\ch_i^{-1}\ot 1\end{equation}
To verify the relation  \ref{linkingrel} one has $$ad(x_i)(\xi_i\ch_i^{-1})=x_i\xi_i\ch_i^{-1}-g_i\xi_i\ch_i^{-1}g_i^{-1}x_i=(x_i\xi_i-\xi_ix_i)\ch_i^{-1}=(1-g_i\ch_i^{-1})$$

Consider $\mathcal{D'}=(G \times \widehat{G},\;(a_i)_{1 \leq i \leq 2\theta},\;(\mu_i)_{1 \leq i \leq 2\theta},\;(b_{ij})_{1 \leq i,j \leq 2\theta})$ where $a_i=g_i$, $a_{\teta+i}= \ch_i^{-1}$  and $\mu_i=\chi_ia_i^{-1}$, $\mu_{\theta+i}=\chi_i^{-1}a_i$ for all $1 \leq i \leq \teta$.

The matrix $(b_{ij})$ consists of two diagonal copies of the matrix $a_{ij}$. It can easily be verified that $\mc{D'}$ is also a datum of finite Cartan type associated to the abelian group $G \times \widehat{G}$. (The character group of $G\times \widehat{G}$ is  identified with $\widehat{G}\times G$.)

Define the linking parameters $\lam$ given by \\$\lam_{ij}=\begin{cases}1,\;\;\;\;\;\;\;j=i+\teta\\0,\;\;\;\;\;\;\;j \neq i+\teta\end{cases}$.
%-q_{ii}^{-1}

If the generating variables of $U(\mc{D'},\;\lam)$ are denoted by $z_i$ then define $$\phi: U(\mc{D'},\;\lam) \ra D(A)$$ by $$\phi(g\chi)=g\chi,\;\;\phi(z_i)=x_i,\;\;\phi(z_{\theta+i})=\xi_i\chi_i^{-1}$$ for all $g\in G$ and $\chi \in \widehat{G}$ and for all $1 \leq i \leq \teta$.

Relations (3.3)-(3.8) show that $\phi$ is a well defined algebra map and relations (3.9)-(3.10) imply that $\phi$ is a Hopf algebra map. In the next corollary it is proved that $\phi$ induces an isomorphism of Hopf algebras $\phi: u(\mc{D'},\;\lam)\ra D(A)$. (See also \cite{MBe}.)
\begin{cor}Let $A = u(\mc{D},\;0,\;0)$  be a pointed Hopf algebra with $\mc{D}=(G,\;(g_i)_{1 \leq i \leq \theta},\;(\ch_i)_{1 \leq i \leq \theta},\;(a_{ij})_{1 \leq i,j \leq \theta})$ a Cartan datum of finite type .

 1) Let $\mathcal{D'}=(G \times \widehat{G},\;(a_i)_{1 \leq i \leq 2\theta},\;(\mu_i)_{1 \leq i \leq 2\theta},\;(b_{ij})_{1 \leq i,j \leq 2\theta})$ and $\lam$ defined as above. Then $D(A)\cong u(\mc{D'},\;\lam)$.

 2) The quantum double $D(A)$ is generated by $G, \;\widehat{G}$, $x_i,\;\xi_i$, $(i=\overline{1, \; \teta})$ with the defining relations given by those of $A$ , $A^*$ and the relations from Proposition \ref{defning}.
\end{cor}

\begin{proof} If $A$ is of nilpotent type with no linking relations then $\widehat{G}=G(A^*)$ and the last relation of the Proposition \ref{defning} holds for any $\gamma \in \widehat{G}$. Then it follows from Corollary \ref{b*} and the PBW-basis description of $A$ that as an algebra $D(A)$ is generated $\{g,\; \ch,\; x_i,\; \xi_i, \; |\; g \in G,\; \ch \in \widehat{G}, \; 1\leq i\leq \teta\}$. This implies that the above map $\phi$  is surjective. Let $s$ the number of connected components of the Dynkin diagram of the Lie algebra $g$. Since $\mtr{dim}\;U(\mc{D'},\;\lam)=|G|^2\prod_{i=1}^sN_i^{2p_i}=\mtr{dim}\;D(A)$ it follows that $\psi$ is an isomorphism.\end{proof}
%\sub{Reduced data of Cartan type}

Let $\Gamma$ be an abelian group $n \geq 1$, $K_i,\;L_i \in \Gamma$, $\chi_i \in \widehat{\Gamma}$ for all $1\leq i \leq n$, and $(a_{ij})_{1 \leq i,\;j \leq n}$ a Cartan matrix of finite type. A $reduced \; datum \; of \; Cartan\; finite \;type$ was defined in \cite{RS2}.\\ It consists of a datum $\mtc{D}_{\text{red}}=\mtc{D}_{\text{red}}(\Gamma,\; (L_i)_{1 \leq i \leq n}, \;(K_i)_{1 \leq i \leq n}, \;(\chi_i)_{1 \leq i \leq n}, \; (a_{ij})_{1\ \leq i,\;j\leq n})$ such that: $$\chi_j(K_i)\chi_i(K_j)=\chi_i(K_i)^{a_{ij}},$$  $$\chi_i(L_j)=\chi_j(K_i),$$ $$K_iL_i \neq 1, \text{and}\; \chi_i(K_i) \neq 1$$ for all $1\leq i,j \leq n$.

%\sub{Definition of $U(\mc{D}_{\mtr{red}},\;l)$}

Let $\mc{D}_{\mtr{red}}$ be a reduced datum of finite Cartan type and $X$ a Yetter-Drinfeld module over $k[\Gamma]$ with basis $x_1,\cdots, x_n,\;y_1,\cdots,y_n$ where $x_i \in X_{L_i}^{\chi_i^{-1}}$ and $y_i \in X_{K_i}^{\chi_i}$. Let $(l_i)_{1\leq i \leq n}$ be a family of nonzero parameters in $k$.

Let $\Gamma$ acting on the free algebra $k<x_1,\;\cdots, x_n,\; y_1,\; \cdots, y_n>$ by $\gamma x_i=\chi_i^{-1}(\gamma)x_i$ and $\gamma y_i=\chi(\gamma)y_i$, for all $\gamma \in \Gamma$ and $1\leq i \leq n$.

The Hopf algebra $U(\mtc{D}_{\text{red}},\;l)$ is defined \cite{RS2} as the quotient of the smash product $k<x_1,\;\cdots, x_n,\;y_1,\cdots,y_n>\#k[\Gamma]$ modulo the ideal generated by
$$ad_c(x_i)^{1-a_{ij}}(x_j) \; \; \text{for all }1\leq i,\;j\leq n,\; i\neq j$$
$$ad_c(y_i)^{1-a_{ij}}(y_j)\;\; \text{for all} 1\leq i,\;j\leq n, \;i\neq j$$
$$x_iy_j-\chi_j(L_i)y_jx_i-\delta_{ij}l_i(1-K_iL_i)\;\; \text{for all}\; 1\leq i,\;j\leq n$$

%As shown in \cite{RS1}, pp.47, it follows that $$E_iF_j-F_jE_i=\delta_{ij}q_{ii}^{-1}l_i(K_i-L_i^{-1}),$$where $E_i=y_i$, $F_i=x_iL_i^{-1}$.

\begin{example}
\end{example}
This example shows that $D(A)$ is a quotient Hopf algebra of $U(\mtc{D}_{\text{red}},\;l)$ whose representations were studied in \cite{RS2}.

Let $A = u(\mc{D},\;0,\;0)$ with $\mc{D}=(G,\;(g_i)_{1 \leq i \leq \theta},\;(\ch_i)_{1 \leq i \leq \theta},\;(a_{ij})_{1 \leq i,j \leq \theta})$ a Cartan datum of finite type.

Let $\mtc{D}_{\mtr{red}}=\mtc{D}_{{\mtr{red}}}(\Gamma,\; (L_i)_{1 \leq i \leq \teta}, \;(K_i)_{1 \leq i \leq \teta}, \;(\mu_i)_{1 \leq i \leq \teta}, \; (a_{ij})_{1\ \leq i,\;j\leq \teta})$ where $\Gamma=G\times \widehat{G}$, $L_i=g_i $, $K_i= \chi_i^{-1}$ and $\mu_i=\ch_i^{-1}\hat{ g_i}$. ( $\widehat{\Gamma}$ is again identified with $\widehat{G} \times G$.) Let $l_i=\lam_{i,\;i+\teta}=1$ for all $1\leq i \leq n$.
Then from \cite{RS2}, page 27 it follows that $U(\mtc{D}_{red},\;l)=U(\mc{D'},\;\lam)$, thus $D(A)$ is a quotient of $U(\mtc{D}_{red},\;l)$.

%\newpage
\begin{example} \end{example} In the next example we will show that certain quantum doubles can be realized as quotients of two parameter quantum groups. This can be regarded as a generalization of the well known fact (for type $A_1$) that the quantum double of a Taft algebra is a a quotient of $u_{q,\;q^{-1}}(sl_2)$.

Let $C={(a_{ij})}_{1 \leq i,\;j \leq \theta}$ be a Cartan matrix of finite type and $g$ the associated semisimple Lie algebra over $Q$. Let $d_i\in \{1,2,3\}$ be a set of relatively prime positive integers such that $d_ia_{ij}=d_ja_{ji}$ for all ${1 \leq i,\;j\leq \theta}$. Let $r,\;s$ be two rational numbers such that $rs^{-1}$ is a root of unity of odd order $N$ and prime with $3$ if $g$ has components of type $G_2$. One can choose $r,\;s,\;N$ such that $r^N=s^N=1$. Let $r_i=r^{d_i} $ and  $s_i=s^{d_i}$, for $1\leq i \leq \teta$.

Let $<-,->$ be the Euler form of $g$ which is the bilinear form on the root lattice $Q$ defined by $$<i,\;j>:=<\alpha_i,\;\alpha_j>=\begin{cases}d_ia_{ij},\;\;\;\;i<j\\d_i,\;\;\;\;\;\;\;\;i=j\\0,\;\;\;\;\;\;\;\;\;i>j\end{cases}$$

%$f(x) = \begin{cases}1 & -1 \le x < 0\\ \frac{1}{2} & x = 0\\x&0<x\le 1\end{cases}$
%
To the Lie algebra $g$ and the numbers $r,\;s$ one can associate a two parameter quantum group $U:=U_{r,\;s}(g)$ as in \cite{tpr}. $U_{r,\;s}(g)$ is generated by $e_i, \; f_i,\; \omega_i^{\pm 1}, \; {\omega'_i}^{\pm 1}$ subject to the following relations:
\vskip 0.3cm
R1) ${\omega_i}^{\pm 1}{\omega_j}^{\pm 1}= \omega_j^{\pm 1}\omega_i^{\pm 1},\;\;\;\;{\omega'_i}^{\pm 1}{\omega'_j}^{\pm 1}= {\omega'_j}^{\pm 1}{\omega'_i}^{\pm 1},$
\vskip 0.3cm
$\;\;\;\;\;{\omega_i}^{\pm 1}{\omega'_j}^{\pm 1}= {\omega'_j}^{\pm 1}\omega_i^{\pm 1},\;\;\;\;\;\;{\omega_i}^{\pm 1}{\omega_i}^{\mp 1}= {\omega'_i}^{\pm 1}{\omega'_i}^{\mp 1}=1.$
\vskip 0.3cm

R2) $\omega_ie_j{\omega_i}^{- 1}= r^{<j,\;i>}s^{-<i,\;j>}e_j,\;\;\;\;{\omega'_i}e_j{\omega'_i}^{- 1}= r^{-<i,\;j>}s^{<j,\;i>}e_j.$
\vskip 0.3cm
R3) $\omega_if_j{\omega_i}^{- 1}= r^{-<j,\;i>}s^{<i,\;j>}f_j,\;\;\;\;{\omega'_i}f_j{\omega'_i}^{- 1}= r^{<i,\;j>}s^{-<j,\;i>}f_j.$
\vskip 0.3cm
R4) $e_if_j-f_je_i=\delta_{i,\;j}\frac{\omega_i-\omega'_i}{r_i-s_i}.$
\vskip 0.3cm
R5) $\sum_{k=0}^{1-a_{ij}}(\begin{matrix}1-a_{ij}\\k\end{matrix})_{r_is_i^{-1}}c_{ij}^{(k)}e_i^{1-a_{ij}-k}e_je_i^k=0\;\;\text{if}\;\; i\neq j.$
\vskip 0.3cm
R6) $\sum_{k=0}^{1-a_{ij}}(\begin{matrix}1-a_{ij}\\k\end{matrix})_{r_is_i^{-1}}c_{ij}^{(k)}f_i^kf_jf_i^{1-a_{ij}-k}=0\;\;\text{if}\;\; i\neq j.$

where $c_{ij}^{(k)}=(r_is_i^{-1})^{k(k-1)/2}r^{k<j,\;i>}s^{-k<i,\;j>}$, for $i \neq j$ and ${n \choose k}_q$ is the quantum binomial coefficient, see Section \ref{Ap}.
$U_{r,\;s}(g)$ is a Hopf algebra with the comultiplication given by $$\D(\omega_i^{\pm 1})=\omega_i^{\pm 1} \ot \omega_i^{\pm 1},\;\;\; \D({\omega'_i}^{\pm 1})={\omega'}_i^{\pm 1} \ot {\omega'}_i^{\pm 1},$$$$\D(e_i)=e_i\ot 1+\omega_i\ot e_i,\;\;\;\;\D(f_i)=f_i\ot {\omega'}_i+ 1\ot f_i.$$ The counit is given by $$\eps(\omega_i^{\pm 1})=1,\;\;\eps({\omega'_i}^{\pm 1})=1,\;\; \eps(e_i)=\eps(f_i)=0$$ and the antipode is given by $$S(\omega_i^{\pm 1})=\omega_i^{\mp 1},\;\; S({\omega'_i}^{\pm 1})={\omega'_i}^{\mp 1},\;\;\;S(e_i)=-{\omega_i}^{-1}e_i,\;\;S(f_i)=-f_i{\omega'_i}^{-1}.$$
Let $G=\prod_{i=1}^{\teta}\mathbb{Z}_{N}$ and $g_1, g_2,\;\cdots, g_{\theta}$ be generators of each component of the product. Define $\chi_i\in \widehat{G}$ by $\chi_i(g_j)=r^{<i,\;j>}s^{-<j,\;i>}$. It can be checked that $\chi_i$ are well defined and $\mc{\widetilde{D}}=( G,\;{({g_i})}_{1 \leq i \leq \teta},\; (\chi_i)_{1 \leq i \leq \teta},\;(a_{ij})_{1 \leq i,j \leq \teta})$ is a Cartan datum of finite type.
Let $A=u(\mc{D}):=u(\mc{D},\;0,\;0)$. We will show that $D(A)$ is a quotient of the two parameter quantum group $U_{r,\;s}(g)$. Define $\psi: U_{r,\;s}(g) \ra u(\mc{D})$ by $\psi(e_i)=\frac{1}{{(s_i-r_i)}^{\frac{1}{2}}}x_i$, $\psi(f_i)=\frac{1}{{(s_i-r_i)}^{\frac{1}{2}}}\xi_i$, $\psi(\omega_i)=g_i$, $\psi(\omega'_i)=\chi_i$.  Using formula \ref{exp} from Appendix it can be checked that $\psi $ is well defined. The relations  R5) and R6) are sent to $0$ by $\phi$ since $ad_{_{A}}(x_i)^{1-a_{ij}}(x_j)=0$ and  respectively $ad_{_{A^*}}(\xi_i)^{1-a_{ij}}(\xi_j)=0$. It can be checked that $\psi$ is a Hopf algebra map. Clearly the ideal of $U$ generated by $<e_i^N,\;f_i^N,\;{\omega_i}^N-1, \; {\omega'_i}^N-1>$ is contained in the kernel of $\psi$. A dimension argument implies that $\mr{ker}(\psi)=<e_i^N,\;f_i^N,\;{\omega_i}^N-1, \; {\omega'_i}^N-1>$.

%One define the operators $ad(\xi_i)$ on $D(A)$ given as left multiplication by $\xi_i$ on the module $D(A)_{ad}$. From \ref{part*} it follows that $ad(\xi_i)(z)=\xi_z-\chi_iz\chi_i^{-1}\xi_i$. In particular $ad(\xi_i)(\xi_j)=\xi_i\xi_j-\chi_i(g_j)\xi_j\xi_i$ for any $1 \leq i,\;j \leq \theta$.
%Let $\Gamma$ be the subgroup of $\widehat{G}$ generated by $\chi_i;\;1\leq i \leq \theta$. Let $V$ the vector space spanned by $\xi_i \;1\leq i \leq \theta$. $V$ can be seen as a $\Gamma$ braided module letting $\xi_i \in V_{\chi_i}^{\hat{g_i}}$ where $\hat(g_i)(\chi)=\chi(g)$ is the evaluation character of $\Gamma$ corresponding $g_i$.
\section{Braided Hopf algebras}\label{braided}
Let $H$ be a finite dimensional Hopf algebra and $R \in ^H_H\mtc{YD}$ be a Yetter-Drinfeld module over $H$.

Recall that $R$ is called a braided Hopf algebra in $^H_H\mtc{YD}$ if it is an algebra and coalgebra such that the comultiplication and counit are morphisms in $^H_H\mtc{YD}$.

Let $A$ and $H$ be Hopf algebras and $p:A\ra H$ and $j:H\ra A$ Hopf algebra homomorphisms such that $pj=\mtr{id}_H$
Let
\bn{equation}
R:=A^ {\;co\; H}=\{a \in A\;|\;(\mtr{id}\ot p)\D(a)=a\ot 1\}
\end{equation}
Then $R$ is a braided Hopf algebra in $^H_H\mtc{YD}$ with the following structures:
\begin{enumerate}
\item $H$ acts on $R$ via the adjoint action.
\item The coaction of $H$ is $(p\ot \mtr{id})\D$.
\item $R$ is a subalgebra of $A$.
\item The comultiplication on $R$ is given by $\D_{ _R}(r)=r_1j pS(r_2)\ot r_3\in R \ot R$.
\end{enumerate}
For all $r\in R$ one has $p(r)=\eps(r)1_H$. This can be seen applying $m\dot(p\ot S_H)$ to the identity $r_1\ot p(r_2)=r\ot 1$.

Define $\nu : A \ra R$ by $\nu(a)=a_1j pS(a_2)$. Then $\nu(ab)=a_1\nu(b)j pS(a_2)$ and $\nu(aj(h))=\nu(a)j(h)$ for all $a,b \in A$ and $h \in H$. It can be proved that $\nu$ is a coalgebra map and it induces a coalgebra isomorphism $\nu: A/Aj(H)^+\cong R$ \cite{AS3}.
Thus $\nu^*:R^* \ra A^*$ is an algebra embedding and \bn{equation}\nu^*(R^*)=\{f \in A^*\;| \;f(ai(h))=f(a)\eps(h)\;\text{for all}\; a \in A, \; h\in H\}\end{equation}
\sub{} \label{isom}The map $\phi : A \ra R\# H$ given by $a \mapsto \nu(a_1)\# p(a_2)$ is an isomorphism of algebras with the inverse given by $r\# h\mapsto rj(h)$.
\sub{Dual Hopf algebra} One has that $p^* :H^*\ra A^*$ and $j^* :A^*\ra H^*$ are Hopf algebra homomorphisms such that $j^*p^*=\mtr{id}_{H^*}$ Then
\bn{eqnarray*}
{A^*}^{\;co \;H^*}   : & \! = \! & \{f \in A^*\;|\;(\mtr{id}\ot j^*)\D(f)=f\ot \eps\} =\\ & =& \{f \in A^*\;|\; f(aj(h))=f(a)\eps(h)\;\text{for all}\; a \in A, \; h\in H\}= \\ & =& \nu^*(R^*)
\end{eqnarray*}

Thus $A^* \cong R^*\# H^*$ via $f \mapsto f_1p^*j^*(Sf_2)\# j^*(f_3)$ with the inverse given by $r^*\#f \mapsto \nu^*(r^*)p^*(f)$.
\sub{}\label{identif} Under  this isomorphism one has that
\bn{eqnarray*}(r^*\#f)(r\#h)& \! = \! &(\nu^*(r^*)p^*(f))(rj(h))=\nu^*(r^*)(r_1j(h_1))p^*(f)(r_2j(h_2))\\ & = & \nu^*(r^*)(r_1)p^*(f)(r_2j(h))=r^*(r_1)(f)(p(r_2)h)\\ & = & r^*(r)f(h)\end{eqnarray*}
\sub{} \label{duality}By duality one has that $A=A^{**}\cong R^{**}\# H^{**}=R\# H$ and it can be checked that the isomorphism obtained in such a way is just the isomorphism from \ref{isom}.

For all $r \in R$ and $h \in H$ it follows that $S^{-1}j(h_2)rj(h_1)\in R$ since \bn{eqnarray*}((\mtr{id}\ot p)\D)(S^{-1}j(h_2)rj(h_1))& \! = \! &S^{-1}j(h_4)r_1j(h_1)\ot p(S^{-1}j(h_3)r_2j(h_2))\\ & = & S^{-1}j(h_4)rj(h_1)\ot p(S^{-1}j(h_3)j(h_2)\\ & = & S^{-1}j(h_1)rj(h_2)\ot 1\end{eqnarray*}
\sub{Definition}Let $A$ be a finite dimensional Hopf algebra. An element $z \in A$ is called a $left\;integral\;of \;A$ (respectively $right\;integral\;of\; A$) if $az=\eps(a)z$ (respectively $za=\eps(a)z$) for all $a \in A$. The space of left (resp.\ right) integrals of $A$ is a one dimensional ideal $\int_A^l$ (resp.\ $\int_A^r$) of $A$ and $S(\int_A^l)=\int_A^r$ where $S$ is the antipode of $A$ (see \cite{montg}).

If $z \in \int_A^l$ is a nonzero left integral of $A$, then there is a unique grouplike element $\gamma \in G(A^*)$, called $the\; distinguished \;grouplike\; element\; of\; A^*$ such that $za=\gamma(a)z$, for all $a \in A$. If $z' \in \int_A^r $ then $az'=\gamma^{-1}z'$, for all $a \in A$.

If $\lam \in \int_{A^*}^r$ is nonzero, then there exists a unique grouplike element $g \in G(A)$ such that $f\lam=f(g)\lam$ for all $f \in A^*$. The element $g \in G(A)$ is \\ called $the \;distinguished \;grouplike\; element \;of \;A$.
\subsection{Left integrals in $A$}\label{leftint} Let $x$ be a left integral of $R$ and $\Lam$ be a left integral of $H$. Then $\Lam x= \Lam_1.x\# \Lam_2$ is a left integral of $A$. Clearly $h(\Lam x)=\eps(h)\Lam x$ and $r(\Lam x)=\Lam_1 ((S\Lam_2).r)x=\Lam_1\eps(S(\Lam_2).r)x=\eps(r)\Lam x$ for all $h \in H$ and $r\in R$.

\sub{} \label{onintegral}There is $\gamma \in G(H^*)$ such that $h.x=\gamma(h)x$ for all $h \in H$. Indeed%j(h_1)xj(Sh_2)
\bn{eqnarray*}
r(h.x)=r(j(h_1)xj(Sh_2))& \! = \! &  j(h_3)(S^{-1}j(h_2)rj(h_1))xS(h_4)\\ & = & j(h_3)\eps(S^{-1}j(h_2)rj(h_1))xS(h_4)\\ & = &\eps(r)(j(h_1)xj(Sh_2) ) \end{eqnarray*}
Thus $h.x$ is an integral in $R$ and since the space of integrals is one dimensional it follows that $h.x=\gamma(h)x$ for some $\gamma \in G(H^*)$.

\subsection{The distinguished grouplike element of $A^*$}\label{distingA*} Let $\al_{ _R}$ and $\al_{ _H}$ be the distinguished grouplike elements of $R$ and $H$. Then
$xr=\al_{ _R}(r)x$ and $\Lam h=\al_{ _H}(h)\Lam$ for all $r\in R$ and $h\in H$. Let $\al_{ _A}$ be the distinguished grouplike element of $A$. Then $\al_{ _A}$ is given by the following equation: $(\Lam x)(r\#h)=\al_{ _A}(r\# h)(\Lam x)$.

On the other hand
\bn{eqnarray*}
( \Lam x)(r\#h) & \! = \! & \al_R(r)(\Lam x h)=\al_R(r)\Lam h_2 (S^{-1}(h_1).r)\\ & = & \al_{ _R}(r)\al_{ _H}(h_2)\gamma^{-1}(h_1)=\al_{ _R}(r)(\gamma^{-1}\al_{ _H})(h)
\end{eqnarray*}
Using \ref{identif} it can be shown that $\al_{ _A}=\al_{ _R}\# \gamma^{-1}\al_H$.

\subsection{Right integral in $A^*$} \label{rightint}If $t$ is a right integral in $R^*$ and $\lam$ a right integral in $H^*$ it can be similarly checked that $t\#\lam$ is a right integral in $A^*$ ( see also \cite{MFS}).

Similarly to \ref{onintegral} it can be proved that there is a group like element $g \in H$ such that $f.t=f(g)t$ for all $f \in H^*$.
Indeed
\bn{eqnarray*}
(f_1ts(f_2))r^*  & \! = \! & f_1t(Sf_4r^*S^2f_3)Sf_2\\ & = & f_1t \eps(Sf_4r^*S^2f_3)Sf_2=\\ & = & r^*(1)f_1 t S(f_2)
\end{eqnarray*}
for all $r^*\in R^*$ and $f \in H^*$. Thus $f.t$ is an integral in $R^*$ and since the space of integrals is one dimensional it follows that $f.t=f(g)t$ for some $g \in G(H)$.
\subsection{The distinguished grouplike element of $A$}\label{distingA}
 Let $g_{ _R}$ and $g_{ _H}$ be the distinguished grouplike elements of $R^*$ and $H^*$. Thus $r^*t=r^*(g_{ _R})t$ and $h^*\lam=h^*(g_{ _H})\lam$ for all $r^* \in R^*$ and $h^* \in H^*$. It follows that the distinguished grouplike element of $A$ is $g_{ _A}=g_{ _R} \# gg_{ _H}$. Indeed,
 \bn{eqnarray*}
 (r^*\# f)(t\# \lam)= r^*(f_1.t)\#f_2\lam=f_1(g)r^*t\#f_2(g_{ _H})\lam=r^*(g_{ _R})f(gg_{ _H})
 \end{eqnarray*}
 for all $r^* \in R^*$ and $f \in H^*$. Using \ref{isom} and \ref{duality} it follows that $g_{ _A}=g_{ _R} \# gg_{ _H}$.

\subsection{The situation when $R$ is a graded braided Hopf algebra}
Suppose $R=\oplus_{i=0}^NR(i)$ is a graded braided Hopf algebra with $R(0)=k$. Then $R(N)=k$ is the space of left and right integrals in $R$ \cite{AnGr}. Thus $R$ is unimodular and $\al_{ _R}=\eps$.
If  $R=\oplus_{i=0}^NR(i)$ is a graded braided Hopf algebra in $^H_H\mtc{YD}$ with $R(0)=k$ then $R^*=\oplus_{i=0}^NR(i)^*$ is a graded braided Hopf algebra in $^{H^*}_{H^*}\mtc{YD}$. Thus $R^*$ is also unimodular and $g_R=1$. Thus in this situation $\al_{ _A}=\eps \# \gamma^{-1}\al_{ _H}$ and $g_{ _A}=1\# gg_{ _H}$
\subsection{Ribbon elements} \label{ribbon} A Hopf algebra $A$ is called quasitriangular if there is an invertible element $R=\sum x_i \ot y_i \in A\ot A$ such that $\D(a)=R\D(a)R^{-1}$ for all  $a\in A$, and $R$ satisfies the following relations $(\D\ot \mr{id})R=R_{13}R_{23}$, $(\mr{id} \ot \D)R=R_{13}R_{12}$ where $R_{12}=\sum x_i\ot y_i\ot 1$, $R_{13}=\sum x_i\ot 1 \ot  y_i$, $R_{23}=\sum 1\ot x_i\ot y_i$. Let $u=\sum S(y_i)x_i$. Then $uS(u)$ is central in $A$ and is referred to as the Casimir element.

An element $v \in A$ is called $quasi-ribbon$ element of a quasitriangular Hopf algebra $(A,\;R)$ if:
\begin{enumerate}
\item $v^2=c$
\item $S(v)=v$,
\item$\eps(v)=1$,
\item $\D(v)=R\widetilde{R}^{-1}(v\ot v)$ where $\widetilde{R}=\sum y_i \ot x_i$ if $R=\sum x_i \ot y_i$.
\end{enumerate}

If $v $ is central in $A$ then $v$ is called $ribbon\;element$ of $A$ and $(A,\;R,\;v)$ is called $ribbon\; Hopf\; algebra$. Ribbon elements are used to construct invariants of knots and links \cite{RY} \cite{RK}, \cite{RYT}.

The Drinfeld double $D(A)$ of a finite dimensional Hopf algebra $A$ is a quasitriangular Hopf algebra with $R=\sum (1\ot e_i)\ot (f_i \ot 1)$ where $e_i$ and $f_i$ are dual bases of $A$ and $A^*$. Kauffman and Radford provided the the following criterion for a Drinfeld double $D(A)$ to be a ribbon Hopf algebra.
\begin{thm}\cite{RK}\label{criterion} Assume $A$ is a finite dimensional Hopf algebra and let $g$ and $\gamma$ be the distinguished group-like elements of $A$ and $A^*$ respectively. Then:

i) $D(A)$ has a quasi-ribbon element if and only if there exist group-like elements $h \in A$, $\delta \in A^*$ such that $h^2=g$ and $\delta^2=\gamma$

ii) $(D(A), R)$ has a ribbon element if and only if there exist $h$ and $\delta$ as in i) such that $$S^2(a)=h(\delta \rh a \lh \delta^{-1})h^{-1}$$
for all $a \in A$.
\end{thm}
\subsection{Condition for $D(A)$ to be ribbon}
According to \ref{criterion} $D(A)$ is a ribbon algebra if and only if there are grouplike elements $\delta_{ _A}$ and $h_{ _A}$ in $A^*$ and $A$, respectively such that $\delta_{ _A}^2=\eps_{ _R} \# \gamma^{-1}\al_{ _H}$ and $h_{ _A}^2=1\# gg_{ _H}$ and
\bn{equation}\label{condit}S^2(a)=h_{_ A}(\delta_{ _A}\rightharpoonup a\leftharpoonup {\delta_{ _A}}^{-1}){h_{_ A}}^{-1}\end{equation} for all $a \in A$.

It is enough to check this conditions on the algebra generators of $A$ namely, $r\in R$ and $j(h)$ with $h \in H$.

\sub{} \label{partsit} Suppose that  there are grouplike elements $\delta \in G(H^*)$ and $h\in G(H)$, respectively such that $\delta^2= \al_{ _H}\gamma^{-1}$ and $h^2=gg_{ _H}$.  Consider $\delta_{ _A}=\eps \# \delta$ and $h_{ _A}=1\# h$. These are grouplike elements of $A^*$  and $A$, respectively and $\delta_{ _A}^2=\eps_{ _R} \# \al_{ _H}\gamma^{-1}$ and $h_{ _A}^2=1\# gg_{ _H}$.

For $a=j(h)$ the condition \ref{condit} becomes

\bn{equation}\label{on H}
S^2(h)=h_{_ H}(\delta\rightharpoonup h\leftharpoonup {\delta}^{-1}){h_{_ H}}^{-1}
\end{equation}

For $a=r$ one has that
 \bn{eqnarray*}
\D_{ _{R \# H}}(r \# 1)& \! = \! & \phi(r_1)\ot\phi(r_2)\\ & = & (\nu(r_1)\# p(r_2))\ot (\nu(r_3)\# p(r_4))\\ & = & (\nu(r_1)\# p(r_2))\ot (\nu(r_3)\# 1)
 \end{eqnarray*}
since $r_1\ot r_2\ot r_3\ot p(r_4)=r_1\ot  r_2 \ot r_3 \ot 1$ for all $r \in R$.

Thus
\bn{eqnarray*}(\eps\# \delta) \rightharpoonup (r\#1) & \! = \! &  (\nu(r_1)\# p(r_2)) <\eps \# \delta, \;\nu(r_3)\# 1>\\ & = &\nu(r_1)\# p(r_2)=\nu(r)\#1\\ & = & r\#1\end{eqnarray*}
Similarly
\bn{eqnarray*}
(r\#1)\leftharpoonup (\eps \# { \delta}^{-1})& \! = \! &(\nu(r_3)\# 1)<\eps \# { \delta}^{-1},\;\nu(r_1)\# p(r_2)>\\ & = &(\nu(r_2)\# 1) {\delta}^{-1}(p(r_1))
\end{eqnarray*}

Thus the condition \ref{condit} for $a=r$ is
\bn{equation}\label{on R}
S^{2}(r)=(h. \nu(r_2))\delta^{-1}(p(r_1))
\end{equation}
for all $r \in R$.

\section{Quantum doubles which are ribbon}\label{Qdwr}

Using the results from the previous section,  we determine the left and right integrals of $A \cong u(\mc{D},\;0,\;0)$ and its distinguished grouplike element. By duality, the integrals of $A^*$ and its distinguished grouplike element are also described. The condition obtained in the previous section for $D(A)$ to be a ribbon algebra will be verified for $A=u(\mc{D},\;0,\;0)$

Consider $H=kG$ for an abelian group $G$ and $V$ be a finite dimensional Yetter-Drinfeld module over the group algebra $kG$. Then $V$ has a basis $(x_i)_{1\leq i \leq \teta}$ with $x_i \in V_{g_i}^{\chi_i}$, where $V_{g_i}^{\chi_i}:=\{gv=\chi_i(g)v,\;\delta(v)=g_i\ot v\}$ and $\delta$ is the comodule structure of $V$.

Suppose that $V\in ^H_H\mtc{YD}$ of finite Cartan type, which means $\chi_i(g_i) \neq 1$ for all $1\leq i\leq \theta$ and there is a Cartan matrix of finite type $(a_{ij})_{1\leq i,j \leq \theta}$ such that $$\chi_j(g_i)\chi_i(g_j)=\chi_i(g_i)^{a_{ij}}$$ for all $1\leq i,j \leq \theta$.

Let $R=B(V)$ the Nichols algebra of a finite dimensional braided vector space $V\in ^H_H\mtc{YD}$ of finite Cartan type. Then $A=u(\mc{D},\;0,\;0)=B(V)\# kG$ and from Corollary \ref{b*} it follows that $A^*=u(\mc{\tilde{D}},\;0,\;0)$.

\begin{prop}\label{int}
Let $\Lam_G=\frac{1}{|G|}\sum_{g\in G}g$ be the integral of $kG$ and $x=\prod_{i=1}^py_i^{N_i-1}$. Then $t_l=\Lam_Gx$ is a left integral of $A$ and $t_r=x\Lam_G$ is a right integral of $A$.
\end{prop}
\begin{proof} Since $x$ is a homogeneous element with maximal degree in $B(V)$, it follows from \cite{AnGr} that $x$ is a left and right integral of $R$. Then \ref{leftint} implies that $\Lam_Gx$ is a left integral in $A$ and \ref{rightint} implies that $x\lam_G$ is a right integral in $A$.
\end{proof}

\begin{prop} \label{disting}The element $\gamma \in G(A^*)$ defined by $\gamma(g)=\prod_{i=1}^p\chi_{\beta_i}^{-(N_i-1)}$ and $\gamma(x_i)=0$ is the distinguished grouplike element of $A^*$.
\end{prop}
\begin{proof} Using \ref{distingA*} the element $\al_{ _A}=\al_{ _R}\# \gamma^{-1}\al_{ _H}$ is distinguished grouplike element of $A^*$.
In the situation $R=B(V)$ and $H=kG$ one has that $\al_{ _R}=\eps_{ _R}$ and $\al_{ _H}=\eps_{ _H}$. On the other hand $\gamma$ is given by the equation $g.x=\gamma(g)x$ for all $g \in G$.
Since $g.y_i=\chi_{_{\beta_i}}(g)y_i$ it follows that $g.x=\prod_{i=1}^p\chi_{_{\beta_i}}(g)^{-(N_i-1)}x$ and thus $\gamma=\prod_{i=1}^p\chi_{_{\beta_i}}^{(N_i-1)}$.
\end{proof}
\begin{prop}\label{int*}
Let $\Lam_{G^*}=\frac{1}{|G|}\sum_{\vartheta \in G}\vartheta$ be the integral of $kG^*$ and $Y=\prod_{i=1}^pY_i^{N_i-1}$. Then $T_l=\Lam_{G^*}Y$ is a left integral of $A^*$ and $T_r=Y\Lam_{G^*}$ is a right integral of $A^*$. Moreover the element $g=\prod_{i=1}^pg_{_{\beta_i}}^{(N_i-1)}$ is the distinguished group like element of $A$.
\end{prop}
\begin{proof} Using Corollary \ref{b*} one has that $A^*=u(\mc{\widetilde{D}},\;0,\;0)$ where $\mc{\widetilde{D}}$ was defined in Section \ref{dual}. Then the Propositions \ref{int}, \ref{disting} applied to $u(\mc{\widetilde{D}},\;0,\;0)$ give the integrals and the distinguished grouplike element of $A^*$.
\end{proof}

\begin{thm} Let $\mtc{D}=(G,\;(g_i)_{1 \leq i \leq \theta},\;(\chi_i)_{1 \leq i \leq \theta},\;(a_{ij})_{1 \leq i,j \leq \theta})$ be a datum of Cartan type and $A = u(\mc{D},\;0,\;0)$ the pointed Hopf algebra associated to it. Assume that the order $N_i$ of $\chi_i(g_i)$ is odd for all $i$ and is prime to $3$ for all i in a connected component of type $G_2$. Then $D(A)$ is a ribbon Hopf algebra.
\end{thm}
\begin{proof}
One has to verify relations \ref{on H} and \ref{on R} from \ref{partsit}.
Using the above notations it follows that $x=\prod_{\al \in \Phi^+}x_{ _\al}^{N_{ _\al}-1}=\prod_{i=1}^py_i^{N_i-1}$ is a left integral in $R=B(V)$. On the other hand from Proposition \ref{disting}\\ $g.x=(\prod_{\al \in \Phi^+}\ch_{ _\al}^{N_{ _\al}-1})(g)$ which means that $$\gamma=\prod_{\al \in \Phi^+}\ch_{ _\al}^{N_{ _\al}-1}.$$

Similarly, Proposition \ref{int*} implies that $t=\prod_{\al \in \Phi^+}Y_{ _\al}^{N_{ _\al}-1}=\prod_{i=1}^pY_i^{N_i-1}$ is a right integral in $R^*$. One has that $\ch . Y=\ch(\prod_{\al \in \Phi^+}g_{ _\al}^{N_{ _\al}-1})Y$ for all $\ch \in kG^*$, which shows that $$g=\prod_{\al \in \Phi^+}g_{ _\al}^{N_{ _\al}-1}.$$
Since $N_{ _\al}$ is odd consider in \ref{partsit}  $\delta=\prod_{\al \in \Phi^+}\ch_{ _\al}^{-\frac{(N_{ _\al}-1)}{2}}$ and $h=\prod_{\al \in \Phi^+}g_{ _\al}^{\frac{(N_{ _\al}-1)}{2}}$.

Condition \ref{on H} is automatically satisfied since $H=kG$ is cocommutative.
Indeed, if $g \in G$ then $S^2(g)=g$ and $h(\delta \rh g \lh \delta^{-1})h^{-1}=\delta(g)\delta^{-1}(g)hgh^{-1}=g$

On the the other hand condition \ref{on R} has to be checked only on a set of algebra generators of $R$, for example $x_i$ with ${1\leq i \leq \teta}$. Since $\D(x_i)=x_i \ot 1+ g_i \ot x_i$ this condition can be written as $h.\nu(x_i)\delta^{-1}(g_i)=S^2(x_i)$. Since $S^2(x_i)= \chi_i(g_i)^{-1}x_i$ and $\nu(x_i)=x_i$ this condition becomes $\ch_i(h)\delta^{-1}(g_i)=\chi_i(g_i)^{-1}$.

For $1 \leq i \leq \teta$, let $s_i$, given by $s_i(\alpha_j)=\alpha_j-a_{ij}\alpha_i$, be the the reflection corresponding to the simple root $\alpha_i$. If $\beta=\sum_{s=1}^{\theta}c_s\alpha_s$ is a root then $$s_i(\beta)=\sum_{s=1}^{\theta}c_ss_i(\alpha_s)=\sum_{s=1}^{\theta}c_s(\alpha_s-a_{is}\alpha_i)=\beta-(\sum_{s=1}^{\theta}c_sa_{is})\alpha_i$$ and therefore \begin{equation}\label{ref}(\sum_{s=1}^{\theta}c_sa_{is})\alpha_i=\beta-s_i(\beta)\end{equation}

Then $$\delta^{-1}(g_i)\chi_i(h)=\prod_{j=1}^p\chi_{_{\beta_j}}^{\frac{(N_j-1)}{2}}(g_i)\chi_i(\prod_{j=1}^pg_{_{\beta_j}}^{\frac{(N_j-1)}{2}})=
\prod_{j=1}^p(\chi_{_{\beta_j}}(g_i)\chi_i(g_{_{\beta_j}}))^{\frac{N_j-1}{2}}$$

Suppose $\beta_j=\sum_{s=1}^{\theta}c_{js}\alpha_s$ with $c_{js}\in \mathbb{Z}_{\geq 0}$, for all $1\ \leq j\leq p$. Then $$\chi_{_{\beta_j}}(g_i)\chi_i(g_{_{\beta_j}})=\prod_{s=1}^{\theta}\chi_s^{c_{js}}(g_i)\chi_i(g_s)^{c_{js}}=\prod_{s=1}^{\teta}(\chi_i(g_i))^{\sum_{s=1}^{\theta}a_{is}c_{js}}$$
 %\subsection{}
 Suppose that $\alpha_i \in J$, the connected component of the Dynkin diagram that contains $\alpha_i$. Without loss of generality one may suppose that $\alpha_{1},\;\cdots,\;\alpha_{\teta_1}$ are the simple roots of $J$ and $\{\beta_{1},\;\cdots,\; \beta_{p_1}\}$ are the corresponding positive roots. It follows that $\chi_{_{\beta_m}}(g_i)\chi_i(g_{_{\beta_m}})=1$ if $m \notin \{1,\;\cdots,\;p_1\}$ since $a_{im}=0$.

Thus $\delta^{-1}(g_i)\chi_i(h)=\prod_{j=1}^{p_1}{\chi_i(g_i)^{
(\sum_{s=1}^{\teta_1}
a_{is}c_{js})\frac{N_i-1}{2}}}$. Since $\chi_i(g_i)^{N_i}=1$ one has that $$\delta^{-1}(g_i)\chi_i(h)=\prod_{j=1}^{p_1}\chi_i(g_i)^{
-\frac{
\sum_{s=1}^{\teta_1}a_{is}c_{js}
}{2}}=\chi_i(g_i)^{
-\frac{\sum_{j=1}^{p_1}
\sum_{s=1}^{\teta_1}a_{is}c_{js}
}{2}}$$

Thus, in order to show that $D(A)$ has a ribbon element one has to check that $\sum_{j=1}^{p_1}
\sum_{s=1}^{\teta_1}a_{is}c_{js}=2$.

Let $\rho_J=\sum_{j=1}^{p_1}\beta_{j}/2$ half sum of the positive roots of the connected component $J$. Using equation \ref{ref} one has $(\sum_{s=1}^{\teta_1}a_{is}c_{js})\alpha_i=\beta_j-s_i(\beta_j)$ for any $1 \leq j \leq p_1$. Therefore $$(\sum_{j=1}^{p_1}
\sum_{s=1}^{\teta_1}a_{is}c_{js})\alpha_i=\sum_{j=1}^{p_1}(\beta_j-s_i(\beta_j))=2(\rho_J-s_i(\rho_J))$$ Since $s_i(\rho_J)=\rho_J-\alpha_i$ one gets that $\sum_{j=1}^{p_1}
\sum_{s=1}^{\teta_1}a_{is}c_{js}=2$.
\end{proof}
\section{Appendix}\label{Ap}

For $n \in \mathbb{N}$ and $q \neq 0$, let $(n)_q=1+q+\cdots+q^{n-1}$ for $n\geq 1$ and $(0)_q=1$. Define $(n)_q!=(1)_q(2)_q\cdots (n)_q$ and let  $${n\choose i}_q=\frac{(n)_q!}{(i)_q!(n-i)_q!}$$ be the quantum binomial coefficients.

Note that if $q \neq 0$ then
\begin{equation}\label{cominv}{n \choose k}_{q^{-1}}={n \choose k}_{q}q^{k(k-n)}\end{equation}
for all $0 \leq k \leq n$. The proof of this is deduced from the equalities: $(n)_{q^{-1}}=q^{-(n-1)}(n)_{q}$ and $(n)_{q^{-1}}!=q^{-n(n-1)/2}(n)_{q}!$.

If $ba=qba$ then $$(a+b)^n=\sum_{i=0}^n{n\choose i}_qa^ib^{n-i}$$ for all $n \in \mathbb{N}$.

Let $A$ be a finite dimensional Hopf algebra. Suppose $x, \; y \in A$ such that $\D(x)=x\ot 1+a\ot x$, $\D(y)=y\ot 1+ b\ot y$ where $a,\;b\in G:=G(A)$. Moreover, suppose that $gxg^{-1}=\chi(g)x$ and $gyg^{-1}=\mu(g)y$ for all $g \in G$ where $\chi, \mu \in \widehat{G}$. Let $z_N=ad(x)^N(y)$. Then \begin{equation}\label{exp}z_N=\sum_{i=0}^N(-1)^i{N\choose i}_{\chi(a)}\chi(a)^{i(i-1)/2}\mu(a)^ix^{N-i}yx^i\end{equation}
The proof of the above formula is by induction on $N$. One has $$z_1=ad(x)(y)=xy-ayS(x)=xy-aya^{-1}x=xy-\mu(a)yx$$ and $$z_{N+1}=ad(x)(z_N)=xz_N-az_Na^{-1}x=xz_{N}-\chi^N\mu(a)z_Nx$$ Thus
\begin{eqnarray*}z_{N+1} \!  & = &  \! \sum_{i=0}^N(-1)^i{N\choose i}_{\chi(a)}\chi(a)^{i(i-1)/2}\mu(a)^ix^{N-i+1}yx^i-
\\ & - &  \sum_{i=0}^N(-1)^i{N\choose i}_{\chi(a)}\chi(a)^{i(i-1)/2}\mu(a)^ix^{N-i}yx^{i+1}\chi(a)^{N}\mu(a)
\\ & = & x^{N+1}y+(-1)^{N+1}\chi(a)^{N(N+1)/2}\mu(a)^Nyx^{N+1}+
\\ & + & \sum_{i=1}^N(-1)^i[{N\choose i}_{\chi(a)}\chi(a)^{i(i-1)/2}
\mu(a)^i+
\\ & + & {N\choose i-1}_{\chi(a)}\chi(a)^{(i-2)(i-1)/2}\mu(a)^{i-1}\chi(a)^{N}\mu(a)]x^{N+1-i}yx^{i}
\\ & = & \sum_{i=0}^{N+1}(-1)^i{N+1\choose i}_{\chi(a)}\chi(a)^{i(i-1)/2}\mu(a)^ix^{N+1-i}yx^i\end{eqnarray*}
since
$${N\choose i}_{\chi(a)}+{N\choose i-1}_{\chi(a)}\chi(a)^{N-i+1}={N+1\choose i}_{\chi(a)}$$

We see that $z_N$ has the same formula as in \cite{AS2}, formula A.8, pp.33. In Lemma A.1, pp. 33 of the same paper it is proved that if $\chi(b)\mu(a)=\chi^{1-r}(a)$ and $z_r=\sum_{i=0}^r\alpha_ix^iyx^{r-i}$ then $\alpha_i$ satisfy the following system:
\begin{equation}\label{s1}\sum_{l \leq i \leq r-h } \alpha_i{i\choose l}_{\chi(a)}{r-i\choose h}_{\chi(a)}\mu(a)^{i-l}\chi(a)^{h(i-l)}=0\end{equation}
\begin{equation}\label{s2}\sum_{u \leq i \leq r-v}\alpha_i{i\choose u}_{\chi(a)}{r-i\choose v}_{\chi(a)}\chi(b)^{r-i-v}\chi(a)^{u(r-i-v)}=0\end{equation}
\vskip -0.3cm
The following lemma and its proof is similar to the Lemma A.1 from \cite{AS2}.
\vskip -0.3cm
\begin{lemma}
Let $A$ be a finite dimensional Hopf algebra.
Suppose $x, \; y \in A$ such that $\D(x)=x\ot 1+a\ot x$, $\D(y)=y\ot 1+ b\ot y$ where $a,\;b\in G:=G(A)$ and $ab=ba$. Moreover, suppose that $gxg^{-1}=\chi(g)x$ and $gyg^{-1}=\mu(g)y$ for all $g \in G$ where $\chi, \mu \in \widehat{G}$. Assume that $\chi(b)\mu(a)=\chi^{1-r}(a)$ for some $r \geq 0$ and let $z=ad(x)^{1-r}(y)$. Then $z$ is a skew primitive element of $A$, $\D(z)=z\ot 1+a^{1-r}b\ot z$.
\end{lemma}
\begin{proof}
It can be shown that $z=\sum_{u=1}^{r}\alpha_ux^uyx^{r-u}$ where $\alpha_u$ are the scalars corresponding to the formula \ref{exp} and they are the same as in \cite{AS2}, Lemma A.1. One has $\D(x^n)=\sum_{i=0}^n{n\choose i}_qx^ia^{n-i}\ot x^{n-i}$ for all $n\geq 0$, where $q=\chi(a)$. Thus
\begin{eqnarray*}\D(z) \! & = & \! \sum_{u=0}^r\alpha_u(\sum_{i=0}^u{u\choose i}_qx^ia^{u-i}\ot x^{u-i})\times (y\ot1+b\ot y)\times \\ & \times  & (\sum_{j=0}^{r-u}{r-u\choose j}_qx^ja^{r-u-j}\ot x^{r-u-j})=
\\ & = &
\sum_{u=0}^r\sum_{i=0}^u\sum_{j=0}^{r-u}\alpha_u{u\choose i}_q{r-u\choose j}_qx^ia^{u-i}yx^ja^{r-u-j}\ot x^{u-i}x^{r-u-j}+
\\ & + & \sum_{u=0}^r\sum_{i=0}^u\sum_{j=0}^{r-u}\alpha_u{u\choose i}_q{r-u\choose j}_qx^ia^{u-i}bx^ja^{r-u-j}\ot x^{u-i}yx^{r-u-j}=
\\ & = & \sum_{u=0}^r\sum_{i=0}^u\sum_{j=0}^{r-u}\alpha_u{u\choose i}_q{r-u\choose j}_q\mu^{u-i}(a)\chi^j(a^{u-i})
x^iyx^{j}a^{r-i-j}\ot x^{r-i-j}+
\\ & + & \sum_{u=0}^r\sum_{i=0}^u\sum_{j=0}^{r-u}\alpha_u{u\choose i}_q{r-u\choose j}_q\chi^j(a^{u-i}b)
x^{i+j}a^{r-i-j}b\ot x^{u-i}yx^{r-u-j}=
\\ & = & \sum_{i+j=r}\alpha_ix^iyx^j\ot 1+
\\ & + & \sum_{0\leq i+j<r}(\sum_{i \leq u \leq r-j}\alpha_u{u\choose i}_q{r-u\choose j}_q
\mu^{u-i}(a)\chi^j(a^{u-i})\;)x^iyx^ja^{r-i-j}\ot x^{r-i-j}+
\\ & + & \sum_{u=0}^r\alpha_ua^rb\ot x^uyx^{r-u} +
\\ & + & \sum_{0 <i+j \leq r}(\sum_{i\leq u \leq r-j}\alpha_u{u\choose i}_q{r-u\choose j}_q\chi^j(a^{u-i}b)
\;)x^{i+j}a^{r-i-j}b\ot x^{u-i}yx^{r-u-j}=
\\ &= & z\ot 1+a^rb \ot z\end{eqnarray*}
The last equality is true since the first term of the above sum is $z\ot 1$ and the third term is $a^rb\ot z$. By  formula \ref{s1} for each $i+j <r$ the coefficient of the $x^iyx^ja^{r-i-j}\ot x^{r-i-j}$ in the second term is zero. Similarly, using that ${u \choose i}_q={u \choose u-i}_q$,  formula \ref{s2} implies that the coefficient of $x^{i+j}a^{r-i-j}b\ot x^{u-i}yx^{r-u-j}$ in the last term is zero.
\end{proof}
\begin{cor}\label{prim}
With the notations from Section \ref{dual}, the element $z=ad_{_{A^*}}(\xi_i)^{1-a_{ij}}(\xi_j)$ is skew primitive in $A^*$: $$\D(z)=z \ot 1+ \chi_i^{1-a_{ij}}\chi_j \ot z$$ for all $1 \leq i,\;j \leq \theta$.
\end{cor}

\begin{prop}\label{newserre}With the notations from Section \ref{qdouble} one has that\\ $ ad_{_{D(A)}}(\xi_i\ch^{-1})^{1-a_{ij}}(\xi_j\ch_j^{-1})=0$  for all $1 \leq i,\;j \leq \theta$.
\end{prop}
\begin{proof}
Note that Proposition \ref{part*} implies that the relations
\begin{equation}\label{needed}\chi\xi_i^s\ch^{-1}=\ch^s(g_i)\xi_i,\;\;(\xi_i\ch)^n=\ch(g_i)^{\frac{n(n-1)}{2}}\xi_i^n\ch^n\end{equation}
hold in $A^*$ for all $\ch \in G(A^*)$, $1\leq i \leq \teta$ and any $n,\;s \geq 0$.

Using formula \ref{exp} the relation $ad_{_{A^*}}(\xi_i)^{1-a_{ij}}(\xi_j)=0$ can be written as:
 $$ \sum_{s=0}^N(-1)^i{N\choose s}_{\chi_i(g_i)}\chi_i(g_i)^{s(s-1)/2}\ch_j(g_i)^s\xi_i^{N-s}\xi_j\xi_i^s=0$$ where $N=1-a_{ij}$. Since $\D_{A^*}(\xi_i)=\xi_i\ot1+\chi_i\ot\xi_i$ one has $\D_{D(A)^*}(\xi_i)=1\ot\xi_i+\xi_i\ot\chi_i$ and thus $\D_{D(A)}(\xi_i\ch_i^{-1})=\ch_i^{-1}\ot \xi_i\ch_i^{-1}+\xi_i\ch_i^{-1}\ot 1$.

Using again formula \ref{exp} one has

\begin{eqnarray*}\! & & \!ad_{_{D(A)}}(\xi_i\ch^{-1})^N (\xi_j\ch_j^{-1})=  \sum_{s=0}^N(-1)^s{N\choose s}_{\chi_i^{-1}(g_i)}\chi_i(g_i)^{-s(s-1)/2}{\ch_i^{-1}(g_j)}^s \cdot\\ & & \cdot{(\xi_i{\ch_i}^{-1})}^{N-s}\xi_j\ch_j^{-1}{(\xi_i\ch_i^{-1})}^s=
\\ & = & \sum_{s=0}^N(-1)^s {N\choose s}_{\chi_i^{-1}(g_i)}\chi_i(g_i)^{-s(s-1)/2}{\ch_i^{-1}(g_j)}^s\cdot
\\ & & \cdot\xi_i^{N-s}\chi_i^{-(N-s)}\xi_j\chi_j^{-1}\xi_i^{s}\ch_i^{-s}{\chi_i^{-1}(g_i)}^{(N-s)(N-s-1)/2+s(s-1)/2}=
\\ & = & \chi_i(g_i)^{(N-N^2)/2}\sum_{s=0}^N(-1)^s {N\choose s}_{\chi_i^{-1}(g_i)}\chi_i(g_i)^{-s(s-1)/2}{\ch_i^{-1}(g_j)}^s\cdot\\ & & \cdot \xi_i^{N-s}\chi_i^{-(N-s)}\xi_j\chi_j^{-1}\xi_i^{s}\ch_i^{-s} \chi_i(g_i)^{-s(s-N)} =\end{eqnarray*}
\begin{eqnarray*}
\\ \!& = & \! \chi_i(g_i)^{(N-N^2)/2}\sum_{s=0}^N(-1)^s{N \choose s}_{\ch_i(g_i)}\ch_i(g_i)^{s(s-n)}\chi_i(g_i)^{-s(s-1)/2}{\ch_i^{-1}(g_j)}^s\cdot \\& & \cdot\xi_i^{N-s}\xi_j\chi_i^{-(N-s)}\xi_i^{s}\chi_j^{-1}\ch_i^{-s} \chi_i(g_j)^{-(N-s)}\ch_j(g_i)^{-s}\chi_i(g_i)^{-s(s-N)}=
\\ & = &   \chi_i(g_j)^{-N}\chi_i(g_i)^{(N-N^2)/2}\sum_{s=0}^N(-1)^s{N \choose s}_{\ch_i(g_i)}\chi_i(g_j)^s\chi_i(g_i)^{-s(s-1)/2}{\ch_i^{-1}(g_j)}^s\cdot \\ & & \cdot\xi_i^{N-s}\xi_j\xi_i^{s} \chi_j(g_i)^{-s}\chi_i(g_i)^{-(N-s)s}\chi_j^{-1}\ch_i^{-N}=
\\ &  = & \chi_j^{-1}\ch_i^{-N}\chi_i(g_j)^{-N} \chi_i(g_i)^{(N-N^2)/2}\sum_{s=0}^N(-1)^s{N \choose s}_{\ch_i(g_i)}\chi_i(g_i)^{-s(\frac{s-1}{2}+N-s)}\chi_j(g_i)^{-s}\cdot \\ & & \cdot \xi_i^{N-s}\xi_j\xi_i^{s}=
\\ & = &  \chi_j^{-1}\ch_i^{-N}\chi_i(g_j)^{-N}\chi_i(g_i)^{(N-N^2)/2}\sum_{s=0}^N(-1)^s{N \choose s}_{\ch_i(g_i)}\chi_i(g_i)^{-s\frac{2N-s-1}{2}}\chi_j(g_i)^{-s}\cdot \\ & & \cdot \xi_i^{N-s}\xi_j\xi_i^{s}=
\\ & = &  \chi_j^{-1}\ch_i^{-N}\chi_i(g_j)^{-N}\chi_i(g_i)^{(N-N^2)/2}\sum_{s=0}^N(-1)^s{N \choose s}_{\ch_i(g_i)}\chi_i(g_i)^{s(s-1)/2}\chi_i(g_i)^{-Ns+s}\chi_j(g_i)^{-s}\cdot \\ & & \cdot \xi_i^{N-s}\xi_j\xi_i^{s}\end{eqnarray*}

If $N=1-a_{ij}$ then $-Ns+s=a_{ij}s$ and $\ch_i(g_i)^{sa_{ij}}\ch_j(g_i)^{-s}=\ch_i(g_j)^s$.

Thus
\begin{eqnarray*}& \! \! & ad(\xi_i\ch^{-1})^{1-a_{ij}}  (\xi_j\ch_j^{-1}) = \chi_j^{-1}\ch_i^{-N}\chi_i(g_j)^{-(1-a_{ij})}\chi_i(g_i)^{(N-N^2)/2}\cdot
\\  & & \cdot \sum_{s=0}^N(-1)^s{N \choose s}_{\ch_i(g_i)}\chi_i(g_i)^{s(s-1)/2}\ch_i(g_j)^s\xi_i^{N-s}\xi_j\xi_i^{s}=
\\ & = & \chi_j^{-1}\ch_i^{-N}\chi_i(g_j)^{-(1-a_{ij})}\chi_i(g_i)^{(N-N^2)/2} ad_{_{A^*}}^{1-a_{ij}}(\xi_i)(\xi_j)=0\end{eqnarray*}

Formula \ref{cominv} for $q=\ch_i(g_i)$ and relations \ref{needed} were used in the above computations.
\end{proof}

\subsection*{Acknowledgments} The author thanks to N. Andruskiewitsch for suggesting the general treatment from Section \ref{braided} and for the useful comments on the paper.
\bibliographystyle{amsplain}
\bibliography{ribbonelements}

\end{document}